\documentclass[11pt, twoside]{amsart}

\title{Noncommutative Shifted Symmetric Functions}
\date{\today}
\author{Robert Laugwitz and Vladimir Retakh}
\address{Department of Mathematics, Rutgers University,
Hill Center, 110 Frelinghuysen Road,
Piscataway, NJ 08854-8019}
\email{robert.laugwitz@rutgers.edu}
\urladdr{https://www.math.rutgers.edu/~rul2/}

\usepackage[all]{xy}
\usepackage{mathabx}
\usepackage{amsmath}
\usepackage{amsfonts}
\usepackage{amsthm}
\usepackage{amssymb}
\usepackage[alphabetic, initials]{amsrefs}
\usepackage[english]{babel}
\usepackage{url}
\usepackage{fancyhdr}
\usepackage{graphicx}
\usepackage[
colorlinks=true,
linkcolor=black, 
anchorcolor=black,
citecolor=black, 
urlcolor=black, 
]{hyperref}
\usepackage{geometry}

\newcommand{\spower}[2]{\left({#1}\!\downarrow \!{#2}\right)}
\newcommand{\opower}[2]{{#1}^{\downarrow {#2}}}
\newcommand{\abinom}{\genfrac{\{}{\}}{0pt}{}}
\newcommand{\dpower}[2]{\left|{#1}\!\downarrow \!{#2}\right|}

\newcommand{\ide}{\operatorname{Id}}

\newcommand{\Sym}{\mathbf{Sym}}

\providecommand{\fr}[1]{\mathfrak{#1}}

\newcommand{\mQ}{\mathbb{Q}}
\newcommand{\mZ}{\mathbb{Z}}
\newcommand{\mN}{\mathbb{N}}

\newtheorem{theorem}{Theorem}[]

\newtheorem{proposition}[theorem]{Proposition}
\newtheorem{corollary}[theorem]{Corollary}
\newtheorem{lemma}[theorem]{Lemma}

\newtheorem{theorem*}{Theorem}

\theoremstyle{definition}
\newtheorem{definition}[theorem]{Definition}

\newtheorem{assumption}[theorem]{Assumption}

\theoremstyle{remark}
\newtheorem{example}[theorem]{Example}
\newtheorem{remark}[theorem]{Remark}

\begin{document}

\begin{abstract}
We introduce a ring of noncommutative shifted symmetric functions based on an integer-indexed sequence of shift parameters. Using generating series and quasideterminants, this multiparameter approach produces deformations of the ring of noncommutative symmetric functions. Shifted versions of ribbon Schur functions are defined and form a basis for the ring. Further, we produce analogues of Jacobi--Trudi and N\"agelsbach--Kostka formulas, a duality anti-algebra isomorphism, shifted quasi-Schur functions, and Giambelli's formula in this setup. In addition, an analogue of power sums is provided, satisfying versions of Wronski and Newton formulas. Finally, a realization of these noncommutative shifted symmetric functions as rational functions in noncommuting variables is given. These realizations have a shifted symmetry under exchange of the variables and are well-behaved under extension of the list of variables.
\end{abstract}
\dedicatory{Dedicated to Grigori Olshanski on occasion of his 70th birthday}
\maketitle

\tableofcontents


\section{Introduction}

Let $\Bbbk$ be a field of characteristic zero. The Hopf algebra of symmetric functions $\Lambda=\varprojlim \Lambda_k$, where $\Lambda_k=\Bbbk[x_1,\ldots,x_n]^{S_n}$, emerges in several areas of mathematics (see e.g. \cite{Mac}). The symmetric functions $\Lambda$ have many applications, especially in representation theory --- for example, the understanding of character formulas for representations of the symmetric groups. The study of $\Lambda$ has a strong combinatorial nature due to the presence of the favourable $\mZ$-basis $\lbrace s_\lambda\rbrace_{\lambda}$ of \emph{Schur functions}, indexed by partitions $\lambda$ \cite{Sch}.

An interesting deformation $\Lambda^*$ of the ring of symmetric functions $\Lambda$ is given in \cite{OO}. This deformation is used to describe the center of $U(\fr{gl}_n)$ as well as the normalized characters of the tower of symmetric groups. The ring $\Lambda^*$ posses a basis of \emph{shifted Schur functions} also indexed by partitions. Their definition is closely related to the \emph{factorial}  Schur functions due to \cites{BL1,BL2}, or the more general multiparameter (or \emph{double}) Schur functions of \cites{Mac2,GG,Mac}.  Note also their connection to flagged double Schur functions \cites{CLL,Las}.

The authors of \cite{Getal} define a Hopf algebra of \emph{noncommutative} symmetric functions $\Sym$ and provide a wealth of analogous results to the commutative case. This approach makes heavy used of \emph{quasideterminants} as a noncommutative analogue of determinants \cites{GR1,GR2,GGRW}. A $\mZ$-basis for $\Sym_{\mZ}$ with good properties is given by the \emph{ribbon Schur functions} $\lbrace R_I\rbrace_I$ indexed by compositions rather than partitions.  A representation theoretic interpretation of this Hopf algebras is given by projective modules over the zero Hecke algebras, where indecomposables are indexed by such partitions \cites{DKLT,KT}.

This paper develops a theory of shifted symmetric functions in a noncommutative setup. Using a multiparameter approach, we adapt the setup of \cites{OO,ORV2,Mol2} while using quasideterminant techniques similar to \cite{Getal}. The results are summarized in Section \ref{summary}.

\subsection{Shifted Symmetric Functions}

In \cite{OO}, Okounkov--Olshanski define a remarkable deformation of the ring of symmetric functions, the ring of \emph{shifted symmetric functions} $\Lambda^*=\varprojlim \Lambda_n^*$. Here, $\Lambda_n^*$ is the ring of polynomials $\Bbbk[x_1,\ldots,x_n]$ stable under the symmetries $(x_i,x_{i+1})\leftrightarrow (x_{i+1}-1,x_i+1)$. The ring $\Lambda^*$ is the free commutative $\Bbbk$-algebra generated by two series: the complete homogeneous shifted symmetric functions $h_1^*,h_2^*,\ldots$, and the elementary shifted symmetric functions $e_1^*,e_2^*,\ldots$. A $\mZ$-basis for $\Lambda_{\mZ}^*=\mZ[h_1^*,h_2^*,\ldots]$ is given by shifted Schur functions $\lbrace s_\lambda^*\rbrace_{\lambda}$, where $\lambda$ are partitions. 
By a Jacobi--Trudi type formula, these functions are defined as
\begin{align}
s_\lambda^*=\det\begin{pmatrix}
h_{\lambda_1}^*&\phi(h_{\lambda_1+1})&\ldots &\phi^{k-1}(h_{\lambda_1+k-1})\\
h_{\lambda_2-1}^*&\phi(h_{\lambda_2})&\ldots &\phi^{k-1}(h_{\lambda_2+k-2})\\
\vdots&\vdots&\ldots&\vdots\\
h_{\lambda_k-k+1}^*&\phi(h_{\lambda_k-k+2})&\ldots &\phi^{k-1}(h_{\lambda_k})
\end{pmatrix},
\end{align}
where $\phi$ is a shift operator, and $\lambda=(\lambda_1\geq \lambda_2\geq \ldots\geq \lambda_k\geq 0)$. The shifted Schur functions also satisfy N\"agelsbach--Kostka and Giambelli formulas. Further, there exists a duality algebra-automorphism $\omega$ satisfying $\omega(s_\lambda^*)=s_{\lambda'}^*$ for the dual partition $\lambda'$ \cite{OO}*{Theorem 4.2}. These shifted Schur functions can further be described using a quotient of determinants
\begin{align}
s_{\lambda}^*(x_1,\ldots,x_n)=\frac{\det\begin{pmatrix}
\opower{(x_1+n-1)}{\lambda_n}&\opower{(x_2+n-2)}{\lambda_n}&\ldots& \opower{x_n}{\lambda_n}\\
\opower{(x_1+n-1)}{(\lambda_{n-1}+1)}&\opower{(x_2+n-2)}{(\lambda_{n-1}+1)}&\ldots& \opower{x_n}{(\lambda_{n-1}+1)}\\
\vdots&\vdots&\ldots&\vdots\\
\opower{(x_1+n-1)}{(\lambda_{1}+n-1)}&\opower{(x_2+n-2)}{(\lambda_{1}+n-1)}&\ldots& \opower{x_n}{(\lambda_{1}+n-1)}
\end{pmatrix}}{\det \begin{pmatrix}
\opower{(x_1+n-1)}{0}&\opower{(x_2+n-2)}{0}&\ldots &\opower{x_n}{0}\\
\vdots&\vdots&\ldots&\vdots\\
\opower{(x_1+n-1)}{(n-1)}&\opower{(x_2+n-2)}{(n-1)}&\ldots &\opower{x_n}{(n-1)}
\end{pmatrix}}.
\end{align}
Here, the falling monomials $\opower{x}{n}=x(x-1)\ldots(x-n+1)$, with $\opower{x}{0}=1$, replace the monomials $x^n$, and $\lambda_l=0$ if $l>k$. These polynomials $s_{\lambda}^*(x_1,\ldots,x_n)$ are stable under the shifted symmetry $(x_i,x_{i+1})\leftrightarrow (x_{i+1}-1,x_i+1)$.
 A notable extension property of the shifted Schur functions $s^*_\lambda$ is that 
\begin{align}
s_\lambda^*(x_1,\ldots, x_n,0)=s_\lambda^*(x_1,\ldots, x_n).
\end{align}
This property ensures that there are morphism of algebras $\Lambda_{n+k}^*\to \Lambda_n^*$ and thus the projective limit $\Lambda^*=\varprojlim \Lambda_n^*$ can be formed. 

Note that for fixed $n$, a change of variables $y_i:=x_i-i$ displays the elements of $\Lambda_n^*$ as 
\emph{symmetric} functions in $(y_1,\ldots, y_n)$. However, in the limit no such substitution is possible, and $\Lambda^*$ defines a genuine deformation of $\Lambda$, which is isomorphic to  the associated graded of $\Lambda^*$ with respect to the polynomial degree grading, cf. \cite{OO}*{Section \S 1}. For interesting recent  applications of $\Lambda^*$ to categorification and nonabelian probability theory see \cite{KLM}.

\subsection{Multiparameter Framework}

The falling monomial $\opower{x}{n}=x(x-1)\ldots(x-n+1)$ equals $(x-a_1)(x-a_2)\ldots(x-a_n)$, where $a_i=i-1$. More generally, one can consider shifted monomials (or powers) 
\begin{align}
\spower{x}{a}^n=(x-a_1)(x-a_2)\ldots(x-a_n),
\end{align}
where $a=(a_i)_{i\in \mZ}$ is a sequence of numbers in $\Bbbk$. This general point of view leads to the definition of \emph{multiparameter} (or \emph{double}) Schur functions  \cite{Mac}*{p. 54ff}. 

In this generality, a ring of $a$-shifted symmetric functions, which we denote by $\Lambda^a$, has been defined in \cite{Oko}*{Remark 2.11}. The duality $\omega$ for shifted symmetric functions is replaced by a more general duality algebra-isomorphism $\omega_a\colon \Lambda^a\to \Lambda^{\hat{a}}$, where $\hat{a}$ are the dual parameters \cite{Mol2}*{Section 2.3}. In \emph{loc.cit.} it is also shown that analogues of the Jacobi--Trudi, N\"agelsbach--Kostka and Giambelli formulas still hold this setup.
Moreover, an extension of the comultiplication structure to the multiparameter framework has been provided in the same paper. 

Super-specializations of multiparameter Schur functions have been studied in \cites{Mol,ORV2}. In particular, this framework has been utilized in \cite{ORV2} to study Frobenius--Schur functions. Applications of these variations of Schur functions exist to equivariant
Schubert classes on Grassmannians (see e.g. references in \cite{Mol2}*{Introduction}). 

\subsection{Noncommutative Symmetric Functions}

Consider $\Sym=\Bbbk\langle S_1,S_2,\ldots\rangle$, the Hopf algebra of noncommutative symmetric functions \cite{Getal}*{Chapter~3}. Note that $\Sym$ is a free graded associative $\Bbbk$-algebra, with generators $S_k$ of degree $k$ --- analogues of complete homogeneous symmetric functions. 
There exist alternative series of free generators for $\Sym$. For example, the elementary symmetric functions $\Sym=\Bbbk\langle \Lambda_1,\Lambda_2,\ldots \rangle$. 
A $\mZ$-basis for this algebra, with coefficients in $\mZ$, is given by so-called \emph{ribbon Schur functions} $R_I$, indexed by compositions $I=(i_1,\ldots, i_n)\in \mN_{\geq 1}^n$. These can be defined using the quasideterminants of \cites{GR1,GR2,GGRW} by
\begin{align}
R_{I}=(-1)^{n-1}\begin{vmatrix}
S_{i_1}&S_{i_1+i_2}&\hdots&S_{i_1+\ldots +i_{n-1}}&\boxed{S_{i_1+\ldots +i_{n}}}\\
1&S_{i_2}&\hdots&S_{i_2+\ldots+i_{n-1}}&S_{i_2+\ldots+i_{n}}\\
0&\ddots&\ddots&\vdots&\vdots\\
\vdots&\hdots&\ddots&S_{i_{n-1}}&S_{i_{n-1}+i_n}\\
0&\hdots&\hdots&1&S_{i_n}
\end{vmatrix}.
\end{align}
Multiplication of ribbon Schur functions follows the formula
\begin{align}
R_{(i_1,\ldots,i_n)}R_{(j_1,\ldots,j_m)}=R_{(i_1,\ldots,i_n,j_1,\ldots,j_m)}+R_{(i_1,\ldots,i_{n-1},i_n+j_1,j_2,\ldots,j_m)},
\end{align}
see \cite{Getal}*{Proposition 3.13}, due to MacMahon \cite{MM} in the commutative case.

A second analogue of Schur functions is given by \emph{quasi-Schur functions}, which are only contained in the skew-field of $\Sym$. These quasi-Schur functions $\breve{S}_\lambda$ are indeed labeled by partitions $\lambda=(\lambda_1\geq \lambda_2\geq \ldots \geq \lambda_n)$, and can also be defined in terms of quasideterminants by
\begin{equation}
\breve{S}_{\lambda}=(-1)^{n-1}\begin{vmatrix}
S_{\lambda_1}&S_{\lambda_2+1}&\ldots&\boxed{S_{\lambda_n+n}}\\
S_{\lambda_1-1}&S_{\lambda_2}&\ldots&S_{\lambda_n-{n}+1}\\
\vdots&\vdots&\ddots&\vdots\\
S_{\lambda_1-{n}+1}&S_{\lambda_2-n+2}&\ldots&S_{\lambda_n}
\end{vmatrix}.
\end{equation}
For these, an analogue of Giambelli's formulas exists \cite{Getal}*{Proposition 3.20}.

Other analogues for Schur functions in the noncommutative case have been constructed (see e.g. \cites{Betal1,Betal2} and reference therein), including the \emph{immaculate basis} which maps to the Schur functions under the natural map $\Sym\to \Lambda$, $S_k\mapsto h_k$.


\subsection{Summary}\label{summary}

This paper develops a theory of noncommutative \emph{shifted} symmetric functions. As in the theory of noncommutative symmetric functions, we use two different approaches to analogues of Schur functions: ribbon Schur functions and quasi-Schur functions. The advantage of ribbon Schur functions is that a basis for a ring of noncommutative shifted symmetric functions is obtained, while certain formulas are more natural in the shifted quasi-Schur function setup. 

First, this paper introduces rings of noncommutative $a$-shifted symmetric functions $\Sym^a$ using an approach through generating series  (Section \ref{sect2.1}--\ref{shiftsection}). These graded algebras depend on an arbitrary sequence of parameter $a=(a_i)_{i\in \mZ}$ in $\Bbbk$ and have different series of free generators: complete homogeneous generators $S_{k;a}$, and elementary symmetric generators $\Lambda_{k;a}$ (as well as power sums $\Psi_{n;a}$, see Section \ref{powersection}). Using the theory of quasideterminants, base change formulas between these generators can be obtained in terms of quasideterminants, giving analogues of the classical Jacobi--Trudi and N\"agelsbach--Kostka formulas (Section \ref{sect2.3}). These formulas suggest a definition of multiparameter (or $a$-shifted) ribbon Schur functions
$R_{I;a}$ which form a $\mZ[a]$-basis for $\Sym^a_{\mZ}$. For these, a shifted MacMahon type multiplication formula 
is derived in Theorem \ref{products}. In addition, $a$-shifted ribbon Schur functions are preserved --- up to shift --- under a duality anti-algebra isomorphism $\omega_a\colon \Sym^a\to \Sym^{\hat{a}}$.

Next, we define a possible analogue of $a$-shifted power sums in Section \ref{powersection}, which can be expressed as a linear combination of $a$-shifted ribbon Schur functions corresponding to hook compositions. These power sum analogues also provide free algebra generators of $\Sym^a$. Analogues of Wronski and Newton formulas hold, and translation matrices between power sums, elementary and complete homogeneous generators are derived in terms of quasi-determinants. A natural Hopf algebra structure on $\Sym^a$ is suggested by requiring that the $a$-shifted power sums are primitive elements (Section \ref{hopf}).

In Section \ref{quasiSchur}, we consider $a$-shifted quasi-Schur functions. These are \emph{not} contained in $\Sym^a$ but in its associated skew-field. However, as their definition closely mimics the commutative case, certain formulas are easier to state. In particular, we proof a version of Giambelli's formula for $a$-shifted quasi-Schur functions in Proposition \ref{giambelli}.

Section \ref{sect3} discusses a specialization of $\Sym^a$ in terms of rational fuctions in a list of noncommuting variables (and their inverses). Here, we restrict to sequences of the form $a_i=i\cdot c+\operatorname{const}$. This way, we obtain deformations of the ring of noncommutative symmetric functions $\Sym$ in the sense that for $c=0$ we recover $\Sym$, and for $c\neq 0$, all deformations are isomorphic to $\Sym^*$ --- a noncommutative version of shifted symmetric functions, corresponding to the case $a_i=i-1$ related to \cite{OO}. This specialization rational expressions  $S_{k;a}(x_1,\ldots,x_n)$ and $\Lambda_{k;a}(x_1,\ldots,x_n)$ can be described in terms of a quotient of quasideterminants (see Theorem \ref{specialize}). A notable property of these specializations is their stability under extension of the list of variables (Section \ref{sect3.2}) which is central in the motivation of the ring of commutative shifted symmetric functions in \cite{OO}. Further, the specializations satisfy that 
\begin{equation}R_{I;a}(x_1,\ldots, x_i,x_{i+1},\ldots,x_n)=R_{I;a}(x_1,\ldots, x_{i+1}-c,x_{i}+c,\ldots,x_n).
\end{equation}
We note that if all variables $x_i$ commute, $S_{k;a}(x_1,\ldots,x_n)$ recovers the commutative shifted complete homogeneous symmetric functions $h_{k;a}(x_1,\ldots,x_n)$, see Proposition \ref{recover}. A similar statement holds for the elementary symmetric generators, but fails for general $a$-shifted ribbon Schur functions. Note that a similar restriction also applies to the ribbon Schur functions without shift.

\subsection{Outlook}

We hope that this paper can serve as the starting point for a theory of noncommutative shifted symmetric functions, analogous to the unshifted theory presented in \cite{Getal}. 

In both the noncommutative case of \cite{Getal}, and the shifted case of \cite{OO} there are more than one analogue of the  power sums in $\Lambda$. We suggest an analogue of noncommutative shifted power sums in Section \ref{powersection}, but there may be other natural analogues of power sums.

The Hopf algebra structure introduced in Section \ref{hopf} needs to be explored further. It does not include general formulas for the coproducts of the generators $\Lambda_{n;a}$, $S_{n;a}$, or the $a$-shifted Ribbon Schur functions yet. We note that in the commutative case, Molev defined dual Littlewood--Richardson polynomials \cite{Mol2}*{Section 4} to compute the coproduct coefficients. 
In addition, we expect applications of the theory presented to the representation theory of the zero Hecke algebra to emerge in forthcoming work. 


It shall also be mentioned that \cite{Oko} defines a larger class of parameter-dependent functions which contains factorial Schur functions as a special case, as well as interpolation Macdonald polynomials. It would be an interesting question to extend the current point of view to study noncommutative interpolation Macdonald polynomials.

\subsection{Acknowledgements}

The authors like to thank A. Ginory, B. Keller, A. Lauve, A. Molev, G. Olshanski, and C. Ozan O\u guz for interesting conversations and comments on the subject of this paper. The research of R. L. is partially supported by the Simons foundation.

\section{Noncommutative Shifted Symmetric Functions}

In this section, we introduce the algebra $\Sym^a$ --- the ring of $a$-shifted symmetric functions, based on a sequence of parameter $a=(a_i)$ for $i\in \mZ$. A main result of the section is the construction of a $\mZ[a]$-basis for $\Sym^a_{\mZ}$ of multiparameter (or $a$-shifted) ribbon Schur functions. 

\subsection{Definitions}\label{sect2.1}

Following \cites{Mol,ORV2}, similar to \cite{Mac}, let $a=(a_n)_{n\in \mZ}$ be a family of complex parameters. We define the \emph{$a$-shifted powers} 
\begin{align}\spower{x}{a}^k:=(x-a_1)\ldots(x-a_k).\end{align}
We denote the \emph{dual parameters} by $\hat{a}$, $\hat{a}_i=-a_{-i+1}$ and observe that $\hat{\hat{a}}=a$. There is a natural \emph{shift operation} $\tau$ on such parameter sequences, defined by $(\tau a)_i=a_{i+1}$.

Define $\mZ[a]$ to be the subring of $\Bbbk$ generated by the elements $a_i$ for all $i$.

\begin{definition}
Define $\Sym^a=\Bbbk\langle S_{1;a},S_{2;a},\ldots\rangle$, the algebra of \emph{noncommutative $a$-shifted symmetric functions}. 
Note that $\Sym^a$ is a free graded associative $\Bbbk$-algebra, with generators $S_{k;a}$ of degree $k$. The generators $S_{k;a}$ are the \emph{$a$-shifted complete homogeneous symmetric functions}.

We often also consider the free associative $\mZ[a]$-algebra $\Sym^a_{\mZ}=\mZ[a]\langle S_{1;a},S_{2;a},\ldots\rangle$. Clearly, $\Sym^a=\Bbbk\otimes_{\mZ[a]}\Sym_{\mZ[a]}^a$.
\end{definition}

\begin{example}
Important cases to consider are when $a_i=0$ for all $i\in \mZ$ which gives the ring of noncommutative symmetric functions $\Sym$, referred to as the \emph{unshifted case}, and the case where $a_i=i-1$ giving a ring of noncommutative shifted symmetric functions $\Sym^*$, analogue to the one studied in \cite{OO} in the commutative case. 
\end{example}

\begin{definition}\label{maindef}
Define the  \emph{$a$-shifted elementary symmetric functions} $\Lambda_{k;a}$. We denote the generating series of the $S_{k;a}$, respectively $\Lambda_{k;a}$, by
\begin{align}
\sigma^{a}(t)&:=1+\sum_{k=1}^{\infty}{\frac{S_{k;a}}{\spower{t}{a}^k}},
\label{Eka}
\\
\lambda^{\hat{a}}(t)&:=1+\sum_{k=1}^{\infty}{\frac{\Lambda_{k; a}}{{\spower{t}{\hat{a}}^k}}},
\label{Lka}
\end{align}
and require the defining equations
\begin{align}\label{powersumequation}
\lambda^{\hat{a}}(-t)\sigma^{a}(t)=\sigma^{a}(t)\lambda^{\hat{a}}(-t)=1.
\end{align}
We adapt the convention that $S_{0;a}=1=\Lambda_{0;a}$. 
\end{definition}

Eq. (\ref{powersumequation}) replaces the unshifted equations from \cite{Getal}*{Section 3.1}.

\begin{example}\label{expl1}
The unshifted symmetric functions are recovered as $S_{k;0}=S_k$ and $\Lambda_{k;0}=\Lambda_k$, for $a=0=(0,0,\ldots)$. 
\end{example}

\begin{example}\label{expl2}
We are especially interested in the case where $a_i=i-1$ for $i\in \mZ$, giving the ring $\Sym^*$ of \emph{noncommutative shifted symmetric functions}. In this case, $\hat{a}_i=i$, and we denote $S_{k;a}=S^*_k$ and $\Lambda_{k;a}=\Lambda^*_k$. This gives noncommutative analogues of the generating series $H^*(u)$ and $E^*(u)$ from \cite{OO}*{12.2}.

Denote the shifted powers in this case by 
$\opower{x}{k}=x(x-1)\ldots (x-k+1)$, and $\opower{x}{0}=1$, which are called \emph{falling powers}.
The corresponding generating series are denoted by
\begin{align}
\sigma^{*}(t)&:=1+\sum_{k=1}^{\infty}{\frac{S^*_{k}}{\opower{t}{k}}},\label{Ekstar}\\
\lambda^{*}(t)&:=1+\sum_{k=1}^{\infty}{\frac{\Lambda^*_{k}}{\opower{t}{k}}}.\label{Lkstar}
\end{align}
\end{example}

The following analogue of \cite{OO}*{Corollary~12.3} holds using these definitions.

\begin{lemma}
The defining identities for the generating series of $\Sym^*$ are
\begin{equation}\label{serieseq}
\sigma^{*}(t)\lambda^{*}(-t-1)=1, \qquad \lambda^{*}(-t-1)\sigma^{*}(t)=1.
\end{equation}
\end{lemma}
\begin{proof}
Equation (\ref{Lka}) is rewritten as 
\begin{align*}
\lambda^{\hat{a}}(-t)&=1+\sum_{k=1}^{\infty}{\frac{\Lambda^*_{k}}{(-t-1)(-t-2)\ldots(-t-k)}}=1+\sum_{k=1}^{\infty}{\frac{\Lambda^*_{k}}{\opower{(-t-1)}{k}}}=\lambda^*(-t-1).
\end{align*}
Hence the defining relation from Eq. (\ref{powersumequation}) implies the claim.
\end{proof}

\begin{example}\label{expl3}
Set $a_i:=i-1/2$, so that $\hat{a}_i=a_i$. This choice gives what can be viewed as a noncommutative analogue of \emph{Frobenius--Schur functions} as considered, in the commutative case, in \cite{ORV2}*{Section \S 2}. However, Frobenius--Schur functions use a super-realization of the ring of symmetric functions, which appears not to be available in the noncommutative literature to the knowledge of the authors.

It was shown in \cite{ORV2}*{Section \S 2}, for $a_i=i-1/2$ in the super-realization, that
$\sigma^a(t)=\sigma(1/t)$, where the latter is the generating series for the unshifted symmetric functions, specialized to $\prod_{i}(1-x_i/t)/(1+y_i/t)$.
\end{example}


\begin{lemma}
As for the $S_{k;a}$, the set of all monomials in
$\lbrace \Lambda_{k;a}\mid k\geq 0\rbrace$ forms a $\mZ[a]$-basis for $\Sym^a_{\mZ}$.
\end{lemma}

We can also adapt the duality isomorphism from \cite{Mol2}*{Section 2.3} to the noncommutative setup, which generalizes the anti-algebra isomorphism from \cite{Getal}*{Section 3}.

\begin{proposition}\label{involution}
There is an anti-algebra\footnote{One can also define an \emph{algebra} isomorphism this way, but in view of Corollary \ref{duality} an anti-algebra morphism is more natural.} isomorphism $\omega_a\colon \Sym^a\to \Sym^{\hat{a}}$, defined by 
\begin{align}
\omega_a(S_{k;a})=\Lambda_{k;\hat{a}},
\end{align}
which satisfies $\omega_{\hat{a}}\omega_a=\ide$.
It follows that 
\begin{align}
\omega_a(\Lambda_{k;a})=S_{k;\hat{a}}.
\end{align}
\end{proposition}
\begin{proof}
Applying $\omega_a$ to the defining relations of Eq. (\ref{powersumequation}), we see that
\begin{align*}
1=\sum_{i,j\geq 0}\frac{\Lambda_{i;\hat{a}}\omega_a(\Lambda_{j;a})}{\spower{t}{a}^i\spower{-t}{\hat{a}}^j}=\sum_{i,j\geq 0}\frac{\Lambda_{i;\hat{a}}\omega_a(\Lambda_{j;a})}{\spower{t}{\hat{\hat{a}}}^i\spower{-t}{\hat{a}}^j}.
\end{align*}
Interchanging $t$ with $-t$, we see that $S_{k;\hat{a}}$ satisfy the same relations as $\omega_a(\Lambda_{k;a})$. However, these relations determine $S_{k;\hat{a}}$ uniquely, whence $S_{k;\hat{a}}=\omega_a(\Lambda_{k;a})$. Now, the equation $\omega_{\hat{a}}\omega_a=\ide$ is easily checked on the algebraic generators $S_{k;a}$, using again that $\hat{\hat{a}}=a$.
\end{proof}

\subsection{Shift Operations}\label{shiftsection}
Recall that for a sequence $a=(a_i)$, we defined the shifted sequence $\tau a$ by
setting $(\tau a)_i=a_{i+1}.$ 
With this notation, we observe the equality
\begin{equation}\label{spowerrec1}
\frac{1}{\spower{t}{\tau a}^k}=\frac{1}{\spower{t}{a}^k}+\frac{a_{k+1}-a_1}{\spower{t}{a}^{k+1}},
\end{equation}
generalizing \cite{OO}*{Eq. (13.3)}.
Inductively, for $l\geq 0$, we obtain the following formulas:

\begin{lemma}\label{spowerreclemma-a}
For any $k,l\geq 0$, and any sequence $a=(a_i)$, we have
\begin{align}\label{spowerrec2-a}
\frac{1}{\spower{t}{\tau^l a}^k}&=\sum_{\nu=0}^l\frac{1}{\spower{t}{a}^{k+\nu}}\abinom{l}{\nu}_{k}^a,\\
\frac{1}{\spower{t}{\tau^{-l} a}^k}&=\sum_{\nu\geq 0}\frac{1}{\spower{t}{a}^{k+\nu}}\abinom{\nu+l-1}{\nu}_{1-(k+\nu)}^{\tau^{k-l}a},\label{spowerrec2-a2}
\end{align}
where 
\begin{align}\label{abinom}
\abinom{l}{\nu}_{k}^a=\sum_{1\leq s_1<\ldots<s_\nu\leq l}\prod_{i=1}^\nu(a_{k+(\nu-i)+s_i}-a_{s_i}).
\end{align}
\end{lemma}
\begin{proof}
For $k\geq 0$, we prove Equation (\ref{spowerrec2-a}) by induction on $l$.
The case $l=0$ is clear. Assume the statement holds for $l\geq 0$.
\begin{align*}
\frac{1}{\spower{t}{\tau^{l+1}a}^{k}}&=\sum_{\nu=0}^l \frac{1}{\spower{t}{\tau a}^{k+\nu}}\abinom{l}{\nu}_{k}^{\tau a}\\
&=\sum_{\nu=0}^l \left(\frac{1}{\spower{t}{a}^{k+\nu}}+\frac{a_{k+\nu+1}-a_1}{\spower{t}{a}^{k+\nu+1}}\right)\abinom{l}{\nu}_{k}^{\tau a}\\
&=\frac{1}{\spower{t}{a}^k}+\sum_{\nu=1}^l\frac{1}{\spower{t}{a}^{k+\nu}}\left(\abinom{l}{\nu}^{\tau a}_k+(a_{k+\nu}-a_1)\abinom{l}{\nu-1}^{\tau a}_k\right)+\frac{a_{k+l+1}-a_1}{\spower{t}{a}^{k+l+1}}\abinom{l}{l}_k^{\tau a}\\
&=\sum_{\nu=0}^{l+1}\frac{1}{\spower{t}{\tau a}^{k+\nu}}\abinom{l+1}{\nu}_k^{a}.
\end{align*}
The last step uses the formula
\begin{align}
\abinom{l}{\nu}^{\tau a}_k+(a_{k+\nu}-a_1)\abinom{l}{\nu-1}^{\tau a}_k=\abinom{l+1}{\nu}_k^{a},
\end{align}
which is directly verified using the definition in Eq. (\ref{abinom}). Equation (\ref{spowerrec2-a2}) is proved similarly.
\end{proof}

We remark that by \cite{ORV2}*{Lemma 2.5}, the coefficients $\abinom{l}{\nu}_k^a$ can be understood as super-realizations of the complete homogeneous symmetric functions, evaluated at certain entries of the parameter sequence $a$. More precisely,
it follows from \cite{ORV2}*{Lemma 2.5} that 
\begin{equation}\label{abinomexplained}
\abinom{n}{\nu}_{k}^a=h_{\nu}(a_{s-1},\ldots,a_{s+k};-a_1,\ldots, -a_{\nu+k-1}).
\end{equation}

\begin{definition}
For $s\in \mZ$, we define the \emph{shifts} $S_{k;a}^{[s]}$, $\Lambda_{k;a}^{[s]}$ by  
\begin{equation}\label{phidef}
\sum_{k\geq 0}\frac{S_{k;a}^{[s]}}{\spower{t}{\tau^{-s}a}^k}=\sum_{k\geq 0}\frac{S_{k;a}}{\spower{t}{a}^k}, \qquad\qquad \sum_{k\geq 0}\frac{\Lambda_{k;a}^{[s]}}{\spower{t}{\tau^{s}\hat{a}}^k}=\sum_{k\geq 0}\frac{\Lambda_{k;a}}{\spower{t}{\hat{a}}^k}.
\end{equation}
\end{definition}

\begin{lemma}\label{shiftauto}
The assignment $\phi^{[s]}(S_{k;a})=S_{k;a}^{[s]}$, extends to an algebra automorphism $\phi^{[s]}$ of $\Sym^a$ such that $\phi^{[s]}(\Lambda_{k;a})=\Lambda_{k;a}^{[s]}.$
\end{lemma}
\begin{proof}
The $\Bbbk$-algebra $\Sym^a$ is free, so we can extend the definition of $\phi^{[s]}$ to monomials in $S_{k;a}$. The defining relations from Eq. (\ref{powersumequation}) are preserved under  $\phi^{[s]}(\Lambda_{k;a})=\Lambda_{k;a}^{[s]}$. This follows as $\widehat{\tau^{-s} a}=\tau^{s}\hat{a}$.
\end{proof}

The following Lemma gives explicit formulas for shifts of the complete homogeneous and elementary noncommutative $a$-shifted symmetric functions.

\begin{lemma}\label{sumlemma} For any $k,s\geq 0$,
\begin{align}\label{posshift}
S_{k;a}^{[s]}&=\sum_{\nu=0}^{k-1}\abinom{s}{\nu}_{k-\nu}^{\tau^{-s}a}S_{k-\nu;a},&\Lambda_{k;a}^{[s]}&=\sum_{\nu=0}^{k-1} \abinom{\nu+s-1}{\nu}_{1-k}^{\tau^{k-\nu}\hat{a}}\Lambda_{k-\nu;a},\\
S_{k;a}^{[-s]}&=\sum_{\nu=0}^{k-1} \abinom{\nu+s-1}{\nu}_{1-k}^{\tau^{k-\nu}a}S_{k-\nu;a}, &\Lambda_{k;a}^{[-s]}&=\sum_{\nu=0}^{k-1}\abinom{s}{\nu}_{k-\nu}^{\tau^{-s}\hat{a}}\Lambda_{k-\nu;a}.\label{negshift}
\end{align}
\end{lemma}
\begin{proof}
These formulas are a direct consequence of Lemma \ref{spowerreclemma-a}.
\end{proof}

\begin{example}The leading coefficients $\abinom{s}{0}_k^a$ are always equal to 1. Further, we see that
\begin{align*}
S_{k;a}^{[1]}&=S_{k;a}+(a_{k-1}-a_{0})S_{k-1;a},&\Lambda_{k;a}^{[-1]}&=\Lambda_{k;a}+(a_{1}-a_{k})\Lambda_{k-1;a}.
\end{align*}
\end{example}

\begin{example}
If $a_i=i-1$ as in the ring of shifted symmetric functions, then
\begin{align*}
\abinom{s}{\nu}_{k}^a=\binom{s}{\nu}\opower{(k+\nu-1)}{\nu},
\end{align*}
and we recover the relation of \cite{OO}*{Eq. (13.5)}
$$\phi^{[s]}(S_k^*)=\sum_{\nu=0}^s\binom{s}{\nu}\opower{(k-1)}{\nu}S_{k-\nu}^*.$$
\end{example}

\subsection{Base-Change via Quasideterminants}\label{sect2.3}

In this section, we express the elementary non\-com\-mu\-ta\-tive $a$-shifted symmetric functions $\Lambda_{k;a}$ in terms of (shifts of) the complete homogeneous series $S_{k;a}$ using quasideterminants, and vice versa. This follows the approach of \cite{Getal}*{Section 3}.

\begin{proposition}For $n\geq 0$,
\begin{align}\label{lineareq}
\sum_{i+j=n}(-1)^jS_{i;a}^{[n-1]}\Lambda_{j:a}=\delta_{n,0}.
\end{align}
\end{proposition}
\begin{proof}
This follows from Eq. (\ref{powersumequation}). Indeed,
\begin{align*}
1&=\sum_{i,j\geq 0}\frac{S_{i;a}\Lambda_{j;}a}{\spower{t}{a}^i\spower{-t}{\hat{a}}^j}=\sum_{i,j\geq 0}(-1)^j\frac{S_{i;a}\Lambda_{j;a}}{\spower{t}{\tau^{-j}a}^{i+j}}=\sum_{i,j\geq 0}(-1)^j\frac{S_{i;a}\Lambda_{j;a}}{\spower{t}{\tau^{n-j-1}\tau^{1-n}a}^{i+j}}.
\end{align*}
Using the $\mZ[a]$-basis $\lbrace 1,1/\spower{t}{\tau^{1-n}a}^1,\ldots, 1/\spower{t}{\tau^{1-n}a}^n,1/t^{n+1},\ldots \rbrace$ for $\Bbbk[1/t]$, we obtain
\begin{align*}
1&=\sum_{k=0}^{n-1}\sum_{i+j=k}(-1)^j\frac{S_{i;a}\Lambda_{j;a}}{\spower{t}{\tau^{n-j-1}\tau^{1-n}a}^{k}}+\sum_{i+j=n}(-1)^j\frac{S_{i;a}\Lambda_{j;a}}{\spower{t}{\tau^{1-n}a}^n}+\text{(higher order terms)}.
\end{align*}
Here, we have used that the expansion of $1/\spower{t}{\tau^{-1}\tau^{1-n}a}$ from Lemma \ref{sumlemma} has leading coefficient one. Now we can expand using Eq. (\ref{spowerrec2-a}) to find
\begin{align*}
1&=\sum_{k=0}^{n}\sum_{i+j=1}\sum_{\nu=0}^{\min\lbrace k,n-k\rbrace }(-1)^j\abinom{n-j-1}{\nu}_k^{\tau^{1-n}a}\frac{S_{i;a}\Lambda_{j;a}}{\spower{t}{\tau^{1-n}a}^{k+\nu}}+\text{(higher order terms)}.
\end{align*}
This implies, by extracting the coefficients of $1/\spower{t}{\tau^{1-n}a}^{n}$ and substituting $i\leftrightarrow i+\nu$,
\begin{align}\label{workingeq}
0&=\sum_{i+j=n}\sum_{\nu=0}^{i-1}\abinom{i-1}{\nu}_{n-\nu}^{\tau^{1-n}a}(-1)^jS_{i-\nu;a}\Lambda_{j;a}\\&=\sum_{i+j=n}\sum_{\nu=0}^{n-1}\abinom{n-1}{\nu}_{i-\nu}^{\tau^{1-n}a}(-1)^jS_{i-\nu;a}\Lambda_{j;a}\\&=\sum_{i+j=n}(-1)^j S_{i;a}^{[n-1]}\Lambda_{j;a}.
\end{align}
Here, we use the following Lemma \ref{difficultcombinatorics} in the second equality. 
\end{proof}

\begin{lemma}\label{difficultcombinatorics}For any series of parameters $a$, and $i,n,\nu\geq 0$,
\begin{equation}
\abinom{i-1}{\nu}_{n-\nu}^a=\abinom{n-1}{\nu}_{i-\nu}^a
\end{equation}
\end{lemma}
\begin{proof}
By \cite{ORV2}*{Lemma 2.5}, cf. Eq. (\ref{abinomexplained}), we find that
\begin{align*}
\abinom{n-1}{\nu}_{i-\nu}^a&=h_\nu(a_{n},\ldots, a_{n+i-\nu-1};-a_1,\ldots, -a_{i-1}), \\
\abinom{i-1}{\nu}_{n-\nu}^a&=h_\nu(a_{i},\ldots, a_{i+i-\nu-1};-a_1,\ldots, -a_{n-1}).
\end{align*}
In the expression for $\abinom{n-1}{\nu}_{i-\nu}^a$, the variables $a_i,\ldots, a_{n-1}$ appear among the first, the symmetric variables, and $-a_i,\ldots, -a_{n-1}$ appear among the anti-symmetric variables. Hence, by the cancellation property of supersymmetric functions, see e.g. \cite{Mac}*{Chapter~3}, these can be removed pairwise, showing equality with $\abinom{i-1}{\nu}_{n-\nu}^a$.
\end{proof}

Eq. (\ref{lineareq}) is analogous to an equation obtained by Macdonald in \cite{Mac}*{Chapter 3.20} in the commutative setup. This equation forms the basis for the definition of multiparameter Schur functions in \cite{ORV2}*{Eq. (3.4)}, and appears in analogues of Jacobi--Trudi and N\"agelsbach--Kostka formulas of shifted Schur functions in \cite{OO}*{Theorem 13.1}. 

In the case where $a_i=0$ for all $i\in\mZ$, Eq. (\ref{lineareq}) recovers the identities in \cite{Getal}*{Eq. (30)} for the noncommutative symmetric functions.
Moreover, we obtain the following expressions in terms of quasideterminants which are analogues of \cite{Getal}*{Eqs. (37), (38)}.

\begin{proposition}\label{transitionprop}For any $n\geq 1$ we find that
\begin{align}
\label{LintermsofS}
\Lambda_{n;a}&=(-1)^{n-1}\begin{vmatrix}
S_{1;a}^{[n-1]}&S_{2;a}^{[n-1]}&\hdots&S_{n-2;a}^{[n-1]}&S_{n-1;a}^{[n-1]}&\boxed{S_{n;a}^{[n-1]}}\\
1&S_{1;a}^{[n-2]}&\hdots&S_{n-3;a}^{[n-2]}&S_{n-2;a}^{[n-2]}&S_{n-1;a}^{[n-2]}\\
0&\ddots&\ddots&\vdots&\vdots&\vdots\\
\vdots&\hdots&\ddots&S_{1;a}^{[2]}&S_{2;a}^{[2]}&S_{3;a}^{[2]}\\
\vdots&\hdots&\hdots&1&S_{1;a}^{[1]}&S_{2;a}^{[1]}\\
0&\hdots&\hdots&0&1&S_{1;a}
\end{vmatrix},
\end{align}
\begin{align}\label{SintermsofL}
S_{n;a}&=(-1)^{n-1}\begin{vmatrix}
\Lambda_{1;a}&\Lambda_{2;a}^{[-1]}&\hdots&\Lambda_{n-2;a}^{[3-n]}&\Lambda_{n-1;a}^{[2-n]}&\boxed{\Lambda_{n;a}^{[1-n]}}\\
1&\Lambda_{1;a}^{[-1]}&\hdots&\Lambda_{n-3;a}^{[3-n]}&\Lambda_{n-2;a}^{[2-n]}&\Lambda_{n-1;a}^{[1-n]}\\
0&\ddots&\ddots&\vdots&\vdots&\vdots\\
\vdots&\hdots&\ddots&\Lambda_{1;a}^{[3-n]}&\Lambda_{2;a}^{[2-n]}&\Lambda_{3;a}^{[1-n]}\\
\vdots&\hdots&\hdots&1&\Lambda_{1;a}^{[2-n]}&\Lambda_{2;a}^{[1-n]}\\
0&\hdots&\hdots&0&1&\Lambda_{1;a}
\end{vmatrix}.
\end{align}\end{proposition}
Note that these quasideterminants evaluate to polynomials in their entries by \cite{GGRW}*{Proposition 2.6}.
\begin{proof}
In order to prove Eq. (\ref{LintermsofS}), we use the system of noncommutative linear equations
\begin{align*}
\Lambda_{n;a}&=S_{1;a}^{[n-1]}\Lambda_{n-1;a}-S_{2;a}^{[n-1]}\Lambda_{n-2;a}\pm\ldots (-1)^{n-2}S_{n-1;a}^{[n-1]}\Lambda_{1;a}+(-1)^{n-1} S_{n;a}^{[n-1]},\\
\Lambda_{n-1;a}&=S_{1;a}^{[n-2]}\Lambda_{n-2;a} \mp\ldots (-1)^{n-3}S_{n-2;a}^{[n-2]}\Lambda_{1;a}+(-1)^{n-2} S_{n-1;a}^{[n-2]},\\
\vdots \qquad &=~\vdots\\
\Lambda_{1;a}&=S_{1;a},
\end{align*}
which is obtained from Eq. (\ref{lineareq}).
Here, the $\Lambda_{k;a}$ are thought of as indeterminants. The result follows from \cite{GR1}*{Theorem 1.8} (see also \cite{GGRW}*{Section 1.6.1}), a noncommutative Cramer's rule. Equation (\ref{SintermsofL}) is proved analogously, using Lemma \ref{shiftauto}, treating instead the $S_{k;a}^{[k-1]}$ as indeterminants in a similar system of noncommutative linear equations, with right multiplication by the coefficient matrix.
\end{proof}

Considering the case $a_i=0$ for all $i\in \mZ$, we recover the equations from \cite{Getal}. 

\begin{example}Special cases give that
\begin{align*}
\Lambda_{2;a}&=-S_{2;a}^{[-1]}+S_{1;a}^{[-1]}S_{1;a},\\
\Lambda_{3;a}&=S_{3;a}^{[-2]}-S_{1;a}^{[-2]}S_{2;a}^{[-1]}-S_{2;a}^{[-2]}S_{1;a}^{[-1]}+S_{1;a}^{[-2]}S_{1;a}^{[-1]}S_{1;a}.
\end{align*}
\end{example}

\subsection{Multiparameter Ribbon Schur Functions}\label{ribbonschur}

Ribbon Schur functions provide a $\mZ$-linear basis of the space $\Sym_{\mZ}$, see \cite{Getal}*{Section 3.2}. 
Here, we adapt the definition of these noncommutative symmetric functions to the $a$-shifted setup.

Let $\Theta$ be a  \emph{ribbon} \cite{Getal}*{Section 3.2}. Such a ribbon is encoded by a \emph{composition} $I=(i_1,\ldots, i_n)\in \mN_{\geq 1}^n$. We think of $\Theta$ as having $i_l$ boxes in the $l$-th row, arranged in a diagram such that the first box of the $(l+1)$-th row lies directly underneath the last box of the $l$-th row. We say that the composition $I$ has \emph{length} $n$.
We denote
\begin{equation}
d_I:=\sum_{j=1}^n i_j
\end{equation}
for the \emph{degree} of the composition $I$. Further denote, for $k=1,\ldots, n-1$,
\begin{align}\label{ribbonshifts}
s_k=\sum_{\nu=k}^{n-1} i_\nu.
\end{align}
In particular, $s_1=d_I-i_n$, and $s_{k+1}>s_{k}\geq n-k$.
With this notation, we define $a$-shifted ribbon Schur functions using an analgue of the classical Jacobi--Trudi formula involving certain shifts in the rows of the matrix.

\begin{definition}\label{Schurdef}
Let $I=(i_1,\ldots, i_n)$ be a composition. Define the \emph{multiparameter} (or \emph{$a$-shifted}) \emph{ribbon Schur function} $R_{I;a}$ by 
\begin{align}
R_{I;a}=(-1)^{n-1}\begin{vmatrix}
S_{i_1;a}^{[s_1]}&S_{i_1+i_2;a}^{[s_1]}&\hdots&S_{i_1+\ldots +i_{n-1};a}^{[s_1]}&\boxed{S_{i_1+\ldots +i_{n};a}^{[s_1]}}\\
1&S_{i_2;a}^{[s_2]}&\hdots&S_{i_2+\ldots+i_{n-1};a}^{[s_2]}&S_{i_2+\ldots+i_{n};a}^{[s_2]}\\
0&\ddots&\ddots&\vdots&\vdots\\
\vdots&\hdots&\ddots&S_{i_{n-1};a}^{[s_{n-1}]}&S_{i_{n-1}+i_n;a}^{[s_{n-1}]}\\
0&\hdots&\hdots&1&S_{i_n;a}
\end{vmatrix}.
\end{align}
\end{definition}

It follows from \cite{Getal}*{Proposition 2.6} that every $R_{I;a}$ is a $\mZ$-linear combination of polynomials of its entries, which are of the form $S_{k;a}^{[l]}$. Now, by Lemma \ref{sumlemma}, all $S_{k;a}^{[l]}$ are $\mZ[a]$-linear combinations of the generators $S_{k;a}$. Thus, $R_{I;a}\in \Sym^a_{\mZ}$ for any composition $I$.

\begin{example}
As in the unshifted case, $R_{I;a}$ is a polynomial in its entries by \cite{Getal}*{Proposition 2.6}, and $R_{(n);a}=S_{n;a}$, where $(n)$ is the length one composition of degree $n$. By  Proposition \ref{transitionprop}, $R_{(1^n);a}=\Lambda_{n;a}$, where $(1^n)=(1,\ldots, 1)$ is the composition of length $n$ containing only ones.
\end{example}

Definition \ref{Schurdef} compares to \cite{OO}*{Equation (3.9)} defining shifted Schur functions, and to \cite{ORV2}*{Eq. (3.4)} defining multiparameter Schur functions in the commutative setting. Note that we used different shifts in the row, which turn out to be more natural for both product formula and duality in the ribbon Schur setup.

\begin{remark}\label{equidistant}
In all Examples \ref{expl1}--\ref{expl3}, $\mZ[a]=\mZ$. This is, in general, the case for any sequence $a$ where $a_i-a_j\in \mZ$ for any $i,j$. We call such a sequence \emph{whole-distant}. In this case, the $a$-shifted ribbon Schur functions form a $\mZ$-basis for $\Sym_\mZ^a$.
\end{remark}

\begin{lemma}\label{RIrecursion}
Let $I=(i_1,i_2,\ldots, i_n)$ be a composition.  Then
\begin{align}\label{schurrec1}
R_{I;a}&=R_{(i_1,\ldots, i_{n-1});a}^{[i_{n-1}]}S_{i_n;a}-R_{(i_1,\ldots, i_{n-2},i_{n-1}+i_{n});a}^{[i_{n-1}]},\\
R_{I;a}&=\begin{cases}S_{i_1;a}^{[d_I-i_n]}R_{(i_2,\ldots,i_n);a}-R_{(i_1+i_2,i_3,\ldots, i_n);a},&n\geq 3\\
S_{i_1;a}^{[i_1]}S_{i_2;a}-S_{(i_1+i_2);a}^{[i_1]}, &n=2.
\end{cases}
\label{schurrec2}
\end{align}
\end{lemma}
\begin{proof}
The proof of Eq. (\ref{schurrec1}) follows that of \cite{Getal}*{Proposition 3.13} closely. Denote by $A$ the matrix used in Definition \ref{Schurdef} to define $R_{I;a}$, and use the same notation for the minor matrix $A^{ij}$ and the quasideterminant $|A|_{ij}$ as in \emph{loc.cit}. We use the homological relations of quasideterminants from \cite{GGRW}*{Theorem 1.4.2} to find that 
\begin{align*}
R_{I;a}&=(-1)^{n-1}|A|_{1n}=(-1)^{n-2} |A^{nn}|_{1,n-1}\cdot |A^{1n}|_{n,n-1}^{-1}\cdot|A|_{nn}\\
&=(-1)^{n-2} |A^{nn}|_{1,n-1}\cdot|A|_{nn}\\
&=(-1)^{n-2} |A^{nn}|_{1,n-1}\left(S_{i_{n};a}-|A^{nn}|_{1,n-1}^{-1}\cdot |A^{n,n-1}|_{1,n}\right)\\
&=R_{(i_1,\ldots, i_{n-1});a}^{[i_{n-1}]}\left(S_{i_n;a}-\left(R_{(i_1,\ldots, i_{n-1});a}^{[i_{n-1}]}\right)^{-1}R_{(i_1,\ldots, i_{n-2},i_{n-1}+i_{n});a}^{[i_{n-1}]}\right)\\
&=R_{(i_1,\ldots, i_{n-1});a}^{[i_{n-1}]}S_{i_n;a}-R_{(i_1,\ldots, i_{n-2},i_{n-1}+i_{n});a}^{[i_{n-1}]}.
\end{align*}
The third equality uses that the middle quasideterminant equals one, and the forth equality uses a row expansion along the last row of the second quasideterminant as in \cite{GGRW}*{Proposition 1.5.1}. The last equations use that the shift $\phi$ is an algebra morphism, and the $a$-shifted ribbon Schur functions are polynomial in their entries. The statement follows using $s_{n-1}=i_{n-1}$.

The second equality is easy in the case $n=2$, and for $n\geq 3$ follows from a similar computation using the other homological relations for the quasideterminants.
\begin{align*}
R_{I;a}&=(-1)^{n-1}|A_{1n}|=(-1)^{n-2}|A|_{11}|A^{1n}|_{21}^{-1}|A^{11}|_{2n}\\
&=(-1)^{n-2}\left( S_{i_1;a}^{[s_1]}-|A^{21}|_{1n}|A^{11}|_{2n}^{-1}\right)|A^{11}|_{2n}\\
&=S_{i_1;a}^{[s_1]}(-1)^{n-2}|A^{11}|_{2n}-|A^{21}|_{1n}\\
&=S_{i_1;a}^{[s_1]}R_{(i_2,\ldots, i_n);a}-R_{(i_1+i_2,i_3,\ldots ,i_n);a},
\end{align*}
this time using the column expansion of $|A|_{11}$ along the first column. We use that $s_1=d_I-i_n$.
\end{proof}

A special case of Eq. (\ref{schurrec1}) directly gives the hook formula
\begin{align}\label{eq52}
\Lambda_{k;a}^{[1]}S_{l;a}=R_{(1^k,l);a}+R_{(1^{k-1},l+1);a}^{[1]}.
\end{align}
This formula is an $a$-shifted analogue of \cite{Getal}*{Eq. (52)}.

\begin{lemma}For any composition $I=(i_1,\ldots, i_n)$ and $1\leq k\leq l\leq n$ denote $i_{k,l}=i_k+\ldots+i_l$. Then
\begin{align}\label{polynomialformula}
R_{I;a}=(-1)^{n-1}S_{i_{1,n};a}^{[s_1]}+\sum (-1)^{k}S_{i_{1,l_1};a}^{[s_1]}S_{i_{l_1+1,l_2};a}^{[s_{l_1+1}]}\ldots S_{i_{l_{k-1}+1,l_k};a}^{[s_{l_{k-1}+1}]}S_{i_{l_k+1,n};a}^{[s_{l_k+1}]},
\end{align}
where the sum is taken over all sequences $1\leq l_1<\ldots<l_k<n$, for $k=1,\ldots n-1$.
\end{lemma}
\begin{proof}
This is a direct consequence of \cite{Getal}*{Proposition 2.6}.
\end{proof}

\begin{example}\begin{align*}
R_{(i,j);a}&=-S_{i+j;a}^{[i]}+S_{i;a}^{[i]}S_{j;a},\\
R_{(i,j,k);a}&=S_{i+j+k;a}^{[i+j]}-S_{i;a}^{[i+j]}S_{j+k;a}^{[j]}-S_{i+j;a}^{[i+j]}S_{k;a}+S_{i;a}^{[i+j]}S_{j;a}^{[j]}S_{k;a}.
\end{align*}
\end{example}

\begin{theorem}[Shifted MacMahon Formula]\label{products}
The set of $a$-shifted ribbon Schur functions $R_{I;a}$, where $I$ is any composition, forms a $\mZ[a]$-basis for $\Sym^a_{\mZ}$.

For two compositions $I=(i_1,\ldots,i_n)$ and $J=(j_1,\ldots, j_m)$ we obtain the product formula
\begin{align}\label{macmahon1}
R_{I;a}^{[d_J-j_m+i_n]}R_{J;a}&=R_{I\cdot J;a}+R_{I\triangleright J;a}, &\text{if } m\geq 2,\\
R_{I;a}^{[i_n]}S_{j_1}&=R_{I\cdot (j_1);a}+R_{I\triangleright (j_1);a}^{[i_n]}, &\text{if }m=1.\label{macmahon2}
\end{align}
Here, $$I\cdot J=(i_1,\ldots, i_n,j_1,\ldots, j_m), \qquad I\triangleright J=(i_1,\ldots, i_{n-1},i_n+j_1,j_2,\ldots, j_m).$$
\end{theorem}

If $a$ is whole-distant (cf. Remark \ref{equidistant}), then the $R_{I;a}$ form a $\mZ$-basis for $\Sym^a_{\mZ}$.

\begin{proof}
We order monomials in $\Sym^a_{\mZ}$ according to the degree lexicographic order. That is, we have $S_{i_1;a}\ldots S_{i_n;a}>S_{j_1;a}\ldots S_{j_m;a}$ if $n>m$ and in case $n=m$ the inequality holds if $(i_1,\ldots, i_n)>(j_1,\ldots, j_m)$ in the lexicographic order. Note that Eq. (\ref{polynomialformula}) implies that 
\begin{align*}
R_{I;a}&=S_{i_1;a}^{[s_1]}S_{i_2;a}^{[s_2]}\ldots S_{i_n;a}+\text{(lower order terms)}\\
&=S_{i_1;a}S_{i_2;a}\ldots S_{i_n;a}+\text{(lower order terms)},
\end{align*}
using Lemma \ref{sumlemma} in the second equality (where all coefficients are in $\mZ[a]$).

Hence we establish a bijection between the $\mZ[a]$-basis of monomials of elementary $a$-shifted symmetric functions $S_{k;a}$ in $\Sym^a_{\mZ}$ and $a$-shifted ribbon functions $R_{I;a}$. This proves the basis claim. 

We prove a more general form of the product formula. For this, we introduce more general shifts of the $R_{I;a}$. We define, for $I=(i_1,\ldots, i_n)$ a composition and $K=(k_1,\ldots,k_n)\in \mZ^n$,
\begin{equation}\label{generalizedribbon}
R_{I;a}^{[K]}=(-1)^{n-1}\begin{vmatrix}
S_{i_1;a}^{[k_1]}&S_{i_1+i_2;a}^{[k_1]}&\hdots&S_{i_1+\ldots +i_{n-1};a}^{[k_1]}&\boxed{S_{i_1+\ldots +i_{n};a}^{[k_1]}}\\
1&S_{i_2;a}^{[k_2]}&\hdots&S_{i_2+\ldots+i_{n-1};a}^{[k_2]}&S_{i_2+\ldots+i_{n};a}^{[k_2]}\\
0&\ddots&\ddots&\vdots&\vdots\\
\vdots&\hdots&\ddots&S_{i_{n-1};a}^{[k_{n-1}]}&S_{i_{n-1}+i_n;a}^{[k_{n-1}]}\\
0&\hdots&\hdots&1&S_{i_n;a}^{[k_{n}]}
\end{vmatrix}.
\end{equation}
Clearly, setting $S=(s_1,\ldots, s_{n-1},0)$ recovers $R_{I;a}^{[S]}=R_{I;a}$.

We claim the more general product formula 
\begin{align}
R_{I;a}^{[K]}R_{J;a}^{[L]}&=R_{I\cdot J;a}^{[K,L]}+R_{I\triangleright J;a}^{[K,l_2,\ldots,l_m]},
\end{align}
for $I=(i_1,\ldots,l_n)$, $J=(j_1,\ldots, j_m)$ compositions and $K=(k_1,\ldots,k_n)$, $L=(l_1,\ldots,l_m)$ integer sequences. The proof is by induction on $m$. 
In the case $m=0$, there is nothing to show, and the case $m=1$ follows from the generalized form of Eq. (\ref{schurrec1}),
\begin{align}
R_{I\cdot j_1;a}^{[K,l_1]}&=R_{I;a}^{[K]}S_{j_1;a}^{[l_1]}-R_{I\triangleright(j_1);a}^{[K]}.
\end{align}
This follows using the same strategy as in the proof of Lemma \ref{RIrecursion}. The special case of $l_1=0$ and $K=(s_1+i_n,s_2+i_n,\ldots,s_{n-1}+i_n,i_n)$ gives Eq. (\ref{macmahon2}).
We further generalize Eq. (\ref{schurrec2}) to
\begin{equation}
R_{i_1\cdot J;a}^{[k_1,L]}=R_{i_1;a}^{[k_1]} R_{J;a}^{[L]}-R_{i\triangleright J;a}^{[k_1,l_2,\ldots, l_m]}.
\end{equation}

Assume the statement holds for all $k<m$ for some fixed $m\geq 1$. 
We derive for $J=(j_1,\ldots,j_m)$, that
\begin{align*}
R_{I;a}^{[K]}R_{J;a}^{[L]}=&R_{I;a}^{[K]}\left(S_{i_1;a}^{[l_1]}R_{(j_2,\ldots, j_{m});a}^{[l_2,\ldots,l_m]}-R_{(j_1+j_2,j_3,\ldots, j_m);a}^{[l_1,l_3,\ldots, l_m]}\right)\\
=&\left(R_{I\cdot (j_1);a}^{[K,l_1]}+R_{I\triangleright (j_1);a}^{[K]}\right)R_{(j_2,\ldots ,j_{m});a}^{[l_2,\ldots, l_m]}-R_{I;a}^{[K]}R_{(j_1+j_2,j_3,\ldots, j_m);a}^{[l_1,l_3,\ldots, l_m]}\\
=&R_{I\cdot J;a}^{[K,L]}+R_{I\cdot(j_1+j_2,\ldots,j_m);a}^{[K,l_1,l_3,\ldots,l_m]}+R_{I\triangleright J;a}^{[K,l_2,\ldots, l_m]}+R_{(i_1,\ldots, i_n+j_1+j_2,j_3,\ldots,j_m);a}^{[K,l_3,\ldots,l_m]}\\
&-R_{I\cdot(j_1+j_2,j_3,\ldots,j_m);a}^{[K,l_1,l_3,\ldots, l_m]}-R_{(i_1,\ldots, i_n+j_1+j_2,j_3,\ldots, j_m);a}^{[K,l_3,\ldots,l_m]}\\
=&R_{I\cdot J;a}^{[K,L]}+R_{I\triangleright J;a}^{[K,l_2,\ldots,l_m]}.
\end{align*}
Here, the first equality uses Eq. (\ref{schurrec2}), and the following steps use the induction step. Eq. (\ref{macmahon2}) follows by considering the case $K=(s_1+d_J-j_m+i_n,s_2+d_J-j_m+i_n,\ldots, s_{n-1}+d_J-j_m+i_n,d_J-j_m+i_n)$, $L=(l_1,\ldots,l_m)$ for $m\geq 2$.
\end{proof}

\begin{example}For $i,j,k,l\geq 1$,
\begin{align*}
R_{(i,j);a}^{[k+j]} R_{(k,l);a}&=R_{(i,j,k,l);a}+R_{(i,j+k,l);a}.
\end{align*}
\end{example}

Thus, $a$-shifted ribbon Schur functions for whole-distant parameter series give a  nonhomogeneous $\mZ$-basis for $\Sym^a_{\mZ}$. A formula for the products of $a$-shifted ribbon Schur functions can also be obtained from the proof of Theorem \ref{products}. In practice, rewriting the closed form in the basis $\lbrace R_{I;a}\rbrace_I$ can be cumbersome, as the easier examples below indicate.

\begin{corollary}For any partitions $I,J$ we have
\begin{align}
R_{I;a}R_{J;a}&=R_{I\cdot J;a}^{[s_1,\ldots, s_{n-1},0,t_1,\ldots,t_m]}+R_{I\triangleright J;a}^{[s_1,\ldots, s_{n-1},0,t_2,\ldots,t_m]},
\end{align}
where $s_k=\sum_{\nu=k}^{n-1}i_\nu$, for $k=1,\ldots, n-1$, and $t_k=\sum_{\nu=k}^{m-1}j_\nu$, for $k=1,\ldots, m-1$.
\end{corollary}

Hence, the leading coefficients of the product of $a$-shifted ribbon Schur functions satisfies the same MacMahon type formula as in the unshifted case, see \cite{Getal}*{Proposition 3.13}, namely
$$R_{I;a}R_{j:a}=R_{I\cdot J;a}+R_{I\triangleright J;a} +\text{(lower degree terms)}.$$

\begin{example}
\begin{align}
\Lambda_{k;a}S_{l;a}&=R_{(1^k,l);a}+R_{(1^{k-1},l+1);a}^{[1]}+(a_1-a_k)\left(R_{(1^{k-1},l);a}+R_{(1^{k-2},l+1);a}\right),\\
S_{k;a}S_{l;a}&=\sum_{\nu=0}^{k-1}\abinom{\nu+k-1}{\nu}_{1-k}^{\tau^{k-\nu}a}\left(R_{(k-\nu,l);a}+S_{(k-\nu+l);a}^{[k-\nu]} \right), &\text{for }k\geq 2,\\
\Lambda_{k;a}\Lambda_{l;a}&=\sum_{\nu=0}^{k-1}\abinom{l}{\nu}_{k-\nu}^{\tau^{-l}\hat{a}}\left(\Lambda_{k-\nu+l;a}+R_{(1^{k-\nu-1},2,1^{l-1});a}\right), &\text{for }l\geq 2.
\end{align}
\end{example}

Let $I$ be a composition, and denote by $I^{\sim}=(j_1,\ldots, j_m)$ the conjugate composition, cf. \cite{Getal}*{Section 3.2}. That is, the ribbon diagram corresponding to $I$ is reflected in the main diagonal. With this notation, we obtain the following analogue of the classical N\"agelsberg--Kostka formula.

\begin{proposition}\label{ribboninL}
For any composition $I$ with conjugate composition $I^\sim=(j_1,\ldots, j_m)$, we have
\begin{equation}\label{dualribbon}
R_{I;a}^{[i_n-1]}=(-1)^{m-1}\begin{vmatrix}
\Lambda_{j_m;a}^{[t_1]}&\Lambda_{j_{m-1}+j_m;a}^{[t_{2}]}&\hdots&\Lambda_{j_2+\ldots +j_{m};a}^{[t_{m-1}]}&\boxed{\Lambda_{j_1+\ldots +j_{m};a}}\\
1&\Lambda_{j_{m-1};a}^{[t_{2}]}&\hdots&\Lambda_{j_2+\ldots+j_{m-1};a}^{[t_{m-1}]}&\Lambda_{j_1+\ldots+j_{m-1};a}\\
0&\ddots&\ddots&\vdots&\vdots\\
\vdots&\hdots&\ddots&\Lambda_{j_{2};a}^{[t_{m-1}]}&\Lambda_{j_{1}+j_2;a}\\
0&\hdots&\hdots&1&\Lambda_{j_1;a}
\end{vmatrix},
\end{equation}
where $t_k=\sum_{\nu=1}^{m-k} j_{\nu}$. 
\end{proposition}
Eq. (\ref{dualribbon}) generalizes Eq. (\ref{SintermsofL}), and provides an alternative may to define $a$-shifted ribbon Schur functions, as polynomials in the $\Lambda_{k;a}$.
\begin{proof}
The proposition follows by applying Jacobi's theorem for the quasi-minors of the inverse matrix, see \cite{Getal}*{Theorem 2.14}, to the matrix used in the definition of $R_{I;a}$. Indeed, for given $d$, consider the $(d+1)\times (d+1)$-matrices 
$A$ and $B$, with
\begin{align*}
B_{ij}=S_{j-i;a}^{[d-i]},&& A_{ij}=(-1)^{d+1-j}\Lambda_{j-i;a}^{[d+1-j]}, 
\end{align*}
where we use the convention that $\Lambda_{-k;a}=S_{-k;a}=0$ for $k>0$. It follows from Eq. (\ref{lineareq}) that $A$ and $B$ are inverse matrices. Note that for this we use that the shift is an algebra morphism.

Now let $I=(i_1,\ldots, i_n)$ be a composition and consider $A$, $B$ of size $d_I+1$. Consider the subset $M=\lbrace i_1+1,i_1+i_2+1,\ldots, i_1+i_2+\ldots+i_{n-1}+1\rbrace$. Applying Jacobi's Theorem for quasideterminants as stated in \cite{Getal}*{Proposition 2.14} to this set $M$, with $L=M$, $i=d_I+1$ and $j=1$ gives $$(-1)^{n-1}|B_{M\cup \lbrace j\rbrace,L\cup \lbrace i\rbrace}|_{j,i}=R_{I;a}^{[i_n-1]}.$$
On the other hand, this $a$-shifted ribbon Schur function equals
\begin{align*}
(-1)^{n-1}\left|A_{P\cup\lbrace i\rbrace, Q\cup\lbrace j\rbrace}\right|_{d_I+1,1}^{-1}&=(-1)^n\left(\left|A_{P\cup\lbrace i\rbrace, Q\cup\lbrace j\rbrace}^{{1,1}}\right|_{d_I+1,d_I+1}\cdot \left|A_{P\cup\lbrace i\rbrace, Q\cup\lbrace j\rbrace}^{d_I+1,1}\right|_{1,d_I+1}^{-1}\right)^{-1}\\
&=(-1)^{n}\left|A_{P\cup\lbrace i\rbrace, Q\cup\lbrace j\rbrace}^{1,d_I+1}\right|_{1,d_I+1}.
\end{align*}
First, we observe that $P\cup\lbrace d_I+1\rbrace=Q\cup\lbrace 1\rbrace$, which implies that this minor of $A$ only has entries $\pm 1$ on the diagonal, and zeros below the diagonal by the definition of $A$. Next, we used expansion along the first column in the second equality. There is only one nonzero entry in this column, and $\pm 1$ on the diagonal of this minor quasideterminant of $A$. 

We need the following combinatorial interpretation of $I^{\sim}$. Consider the set $Q\cup\lbrace d_I+1\rbrace=(\mu_1,\ldots, \mu_m)$ as above in increasing order and define $j_i=m_i-m_{i-1}$, with $m_0=1$. Then $J=(j_m,\ldots, j_1)$. This implies that
$$(-1)^{n}\left|A_{P\cup\lbrace i\rbrace, Q\cup\lbrace j\rbrace}^{1,d_I+1}\right|_{1,d_I+1}=(-1)^{m-1}|C_{1,d_I+1}|_{1,d_I+1},$$
where $C$ is the matrix of the right hand side of  Eq. (\ref{dualribbon}).
\end{proof}

\begin{example}
\begin{align*}
R_{(2,1,1);a}&=\begin{vmatrix}
S_{2;a}^{[3]}&S_{3;a}^{[3]}&\boxed{S_{4;a}^{[3]}}\\
1&S_{1;a}^{[1]}&S_{2;a}^{[1]}\\
0&1&S_{1;a}
\end{vmatrix}=(-1)\begin{vmatrix}
\Lambda_{1;a}^{[3]}&\boxed{\Lambda_{4;a}}\\
1&\Lambda_{3;a}
\end{vmatrix},\\
R_{(1,3,2,1);a}&=(-1)\begin{vmatrix}
S_{1;a}^{[6]}&S_{4;a}^{[6]}&S_{6;a}^{[6]}&\boxed{S_{7;a}^{[6]}}\\
1&S_{3;a}^{[5]}&S_{5;a}^{[5]}&S_{6;a}^{[5]}\\
0&1&S_{2;a}^{[2]}&S_{3;a}^{[2]}\\
0&0&1&S_{1;a}
\end{vmatrix}=(-1)\begin{vmatrix}
\Lambda_{2;a}^{[5]}&\Lambda_{3;a}^{[4]}&\Lambda_{5;a}^{[2]}&\boxed{\Lambda_{7;a}}\\
1&\Lambda_{1;a}^{[4]}&\Lambda_{3;a}^{[2]}&\Lambda_{5;a}\\
0&1&\Lambda_{2;a}^{[2]}&\Lambda_{4;a}\\
0&0&1&\Lambda_{2;a}
\end{vmatrix}.
\end{align*}
\end{example}

\begin{corollary}[Duality]\label{duality}
The anti-algebra isomorphism $\omega_a\colon \Sym^a\to \Sym^{\hat{a}}$ from Proposition \ref{involution} satisfies the equation
\begin{align}
\omega_a\left(R_{I;a}\right)=R_{I^\sim;\hat{a}}^{[j_m-d_J+i_n]},
\end{align}
where $I=(i_1,\ldots,i_n)$ and $I^\sim=(j_1,\ldots, j_m)$.
\end{corollary}
\begin{proof}
After application of $\omega_a$, this follows --- under  use of Proposition \ref{involution} --- 
 from Proposition \ref{ribboninL}. Note that these $a$-shifted ribbon Schur functions are polynomial in their entries and that switching to the opposite algebra in $|A|_{1n}$ reflects the matrix $A$ along the anti-diagonal. Further, we use that $\omega_a(S_{k;a}^{[s]})=\Lambda_{k;\hat{a}}^{[-s]}$, which follows from Lemma \ref{sumlemma}.
\end{proof}

\begin{example}
For $I=(2,2,3,2)$, $I^\sim=(1,3,2,2,1)$, we see that
\begin{align*}
\omega_a\left(R_{I;a}\right)=-\omega_a\left(\begin{vmatrix}
S_{2;a}^{[7]}&S_{4;a}^{[7]}&S_{7;a}^{[7]}&\boxed{S_{9;a}^{[7]}}\\
1&S_{2;a}^{[5]}&S_{5;a}^{[5]}&S_{7;a}^{[5]}\\
0&1&S_{3;a}^{[3]}&S_{5;a}^{[3]}\\
0&0&1&S_{2;a}
\end{vmatrix}\right)=\begin{vmatrix}
S_{1;a}^{[1]}&S_{4;a}^{[1]}&S_{6;a}^{[1]}&S_{8;a}^{[1]}&\boxed{S_{9;a}^{[1]}}\\
1&S_{3;a}&S_{5;a}&S_{7;a}&S_{8;a}\\
0&1&S_{2;a}^{[-3]}&S_{4;a}^{[-3]}&S_{5;a}^{[-3]}\\
0&0&1&S_{2;a}^{[-5]}&S_{3;a}^{[-5]}\\
0&0&0&1&S_{1;a}^{[-7]}
\end{vmatrix}.
\end{align*}
\end{example}

Note that Corollary \ref{duality} is a main justification for the use of the shift parameters $s_1,\ldots,s_n-1$ in the definition of the $R_{I;a}$. In particular, these shifts are determined by requiring a definition of $R_{I;a}$ which is preserved under $\omega_a$ up to shift and recovers $S_{n;a}$, $\Lambda_{n;a}$ as special cases without adjustment, as well as recovering the unshifted ribbon Schur functions $R_I$ if $a=0$. 


\subsection{Noncommutative Shifted Power Sums}\label{powersection}

The hook formula Eq. (\ref{eq52}) suggest, in analogy with \cite{Getal}*{Eq. (53)}, a definition of $a$-shifted power sums. These provide an alternative set of free algebra generators for $\Sym^a$.

\begin{definition}
For $n\geq 0$, we define the \emph{$a$-shifted power sum}
\begin{align}
\Psi_{n;a}&=\sum_{k=0}^{n-1}(-1)^{k}(n-k)\Lambda_{k;a}^{[n-k]}S_{n-k;a}^{[n-k-1]}.
\end{align}
\end{definition}

The definition is justified, in parts, by the following analogue of \cite{Getal}*{Corollary 3.14}, which is a corollary of Eq. (\ref{eq52}).

\begin{corollary}
\begin{align}
\Psi_{n;a}&=\sum_{k=0}^{n-1} (-1)^{k}R_{(1^{k},n-k);a}^{[n-k-1]},\\
\Psi_{n;a}&=\sum_{k=1}^n(-1)^{k-1}k \Lambda_{k;a}^{[n-k]}S_{n-k;a}^{[n-k-1]}.\label{PsiLrel}
\end{align}
\end{corollary}

\begin{example}
\begin{align*}
\Psi_{1;a}&=S_{1;a}=\Lambda_{1;a},\\
\Psi_{2;a}&=2S_{2;a}^{[1]}-\Lambda_{1;a}^{[1]}S_{1;a}=S_{2;a}^{[1]}-\Lambda_{2;a},\\
\Psi_{3;a}&=3S_{3;a}^{[2]}-2\Lambda_{1;a}^{[2]}S_{2;a}^{[1]}+\Lambda_{2;a}^{[1]}S_{1;a}=S_{3;a}^{[2]}-R_{(1,2);a}^{[1]}+\Lambda_{3;a}.
\end{align*}
\end{example}

\begin{proposition}
The $a$-shifted power sums form a set of free generators for $\Sym^{a}$.
\end{proposition}
\begin{proof}
We use the degree lexicographical ordering on monomials in $\lbrace\Lambda_{k;a}\rbrace_{k\geq 1}$, where $\Lambda_{k;a}>\Lambda_{l;a}$ if and only if $k>l$, in $\Sym^a$. That is, when comparing two monomials in $\Lambda_{k;a}$'s, we first compare their degrees (with the grading $\deg \Lambda_{k;a}=k$). If the degrees are equal, we use the lexicographical order. This way, we observe that $$\Psi_{n;a}=\Lambda_{n;a}+\text{(lower order terms)}.$$
This way, we see that monomials in $\lbrace\Psi_{k;a}\rbrace_{k\geq 1}$, and hence the $\Psi_{k;a}$ are free algebra generators as the $\Lambda_{k;a}$ are. 
\end{proof}

We obtain the following $a$-shifted analogue of the noncommutative Wronski and Newton formula, Eq. (\ref{wronski}) and (\ref{newton}).

\begin{lemma}\label{psieq}For $n\geq 1$, we find
\begin{align}\label{otherfundamental}
S_{n;a}^{[n-1]}&=\sum_{k=1}^{n}(-1)^{k-1}\Lambda_{k;a}^{[n-k]}S_{n-k;a}^{[n-1-k]},\\
nS_{n;a}^{[n-1]}&=\sum_{k=0}^{n-1}S_{k;a}^{[n-1]}\Psi_{n-k;a}.\label{wronski},\\
n\Lambda_{n;a}&=\sum_{k=0}^{n-1}(-1)^{n-k-1}\Psi_{n-k;a}^{[k]}\Lambda_{k;a}.\label{newton}.
\end{align}
\end{lemma}
\begin{proof}
Eq. (\ref{otherfundamental}) follows directly from Eq. (\ref{eq52}). We prove Eq. (\ref{wronski}) by induction on $n$. The case $n=1$ is clear. Fix $n\geq 2$ and assume the statement holds for smaller values than $n$. By definition, we find
\begin{align*}
nS_{n;a}^{[n-1]}&=\Psi_{n;a}+\sum_{k=1}^{n-1}(-1)^{k-1}(n-k)\Lambda_{k;a}^{[n-k]}S_{n-k;a}^{[n-k-1]}\\
&=\Psi_{n;a}+\sum_{k=1}^{n-1}(-1)^{k-1}\Lambda_{k;a}^{[n-k]}\sum_{l=0}^{n-k-1}S_{l;a}^{[n-k-1]}\Psi_{n-k-l;a},
\end{align*}
where in the second equality, we use the induction hypothesis. We can identify the coefficient $c_\nu$ of $\Psi_{n-\nu;a}$ as
\begin{align*}
c_{\nu}&=(-1)^{k-1}\Lambda_{k;a}^{[n-k]}+(-1)^{k-2}\Lambda_{k-1;a}^{[n-k+1]}S_{1;a}^{[n-k]}\pm\ldots\pm \Lambda_1^{[n-1]}S_{k-1;a}^{[n-2]}.
\end{align*}
Now repeated application of Eq. (\ref{otherfundamental}) gives that $c_{\nu}=S_{k;a}^{[n-1]}$, using the expansion of $S_{k;a}^{[n-1]}$ in terms of a polynomial in the shifts of the $\Lambda_{l;a}$'s from Eq. (\ref{SintermsofL}).

The proof of Eq. (\ref{newton}) follows a very similar pattern, using Eq. (\ref{PsiLrel}) instead of the definition of $\Psi_{n;a}$.
\end{proof}

We obtain the following translation identities, analogue to \cite{Getal}*{Corollary 3.6}, using quasi-determinants.

\begin{proposition}\label{transitionpower}For any $n\geq 1$ we find that
\begin{align}
\label{PsiintermsofS}
\Psi_{n;a}&=\begin{vmatrix}
1S_{1;a}^{[n-1]}&S_{2;a}^{[n-1]}&\hdots&S_{n-2;a}^{[n-1]}&S_{n-1;a}^{[n-1]}&\boxed{nS_{n;a}^{[n-1]}}\\
1&S_{1;a}^{[n-2]}&\hdots&S_{n-3;a}^{[n-2]}&S_{n-2;a}^{[n-2]}&(n-1)S_{n-1;a}^{[n-2]}\\
0&\ddots&\ddots&\vdots&\vdots&\vdots\\
\vdots&\hdots&\ddots&S_{1;a}^{[2]}&S_{2;a}^{[2]}&3S_{3;a}^{[2]}\\
\vdots&\hdots&\hdots&1&S_{1;a}^{[1]}&2S_{2;a}^{[1]}\\
0&\hdots&\hdots&0&1&1S_{1;a}
\end{vmatrix},\\
\label{PsiintermsofL}
\Psi_{n;a}&=(-1)^{n-1}\begin{vmatrix}
\Lambda_{1;a}^{[n-1]}&2\Lambda_{2;a}^{[n-2]}&\hdots&(n-2)\Lambda_{n-2;a}^{[2]}&(n-1)\Lambda_{n-1;a}^{[1]}&\boxed{n\Lambda_{n;a}}\\
1&\Lambda_{1;a}^{[n-2]}&\hdots&\Lambda_{n-3;a}^{[2]}&\Lambda_{n-2;a}^{[1]}&\Lambda_{n-1;a}\\
0&\ddots&\ddots&\vdots&\vdots&\vdots\\
\vdots&\hdots&\ddots&\Lambda_{1;a}^{[2]}&\Lambda_{2;a}^{[1]}&\Lambda_{3;a}\\
\vdots&\hdots&\hdots&1&\Lambda_{1;a}^{[1]}&\Lambda_{2;a}\\
0&\hdots&\hdots&0&1&\Lambda_{1;a}
\end{vmatrix},\end{align}\begin{align}
\label{SintermsofPsi}
nS_{n;a}^{[n-1]}&=\begin{vmatrix}
\Psi_{1;a}^{[n-1]}&\Psi_{2;a}^{[n-2]}&\Psi_{3;a}^{[n-3]}&\hdots&\Psi_{n-1;a}^{[1]}&\boxed{\Psi_{n;a}}\\
-1&\Psi_{1;a}^{[n-2]}&\Psi_{2;a}^{[n-3]}&\hdots&\Psi_{n-2;a}^{[1]}&\Psi_{n-1;a}\\
0&-2&\Psi_{1;a}^{[n-3]}&\hdots&\Psi_{n-3;a}^{[1]}&\Psi_{n-2;a}\\
\vdots&\ddots&\ddots&\ddots&\vdots&\vdots\\
\vdots&\hdots&\hdots&2-n&\Psi_{1;a}^{[1]}&\Psi_{2;a}\\
0&\hdots&\hdots&0&1-n&\Psi_{1;a}
\end{vmatrix},\\\label{LintermsofPsi}
n\Lambda_{n;a}&=(-1)^{n-1}\begin{vmatrix}
\Psi_{1;a}^{[n-1]}&\Psi_{2;a}^{[n-2]}&\Psi_{3;a}^{[n-3]}&\hdots&\Psi_{n-1;a}^{[1]}&\boxed{\Psi_{n;a}}\\
n-1&\Psi_{1;a}^{[n-2]}&\Psi_{2;a}^{[n-3]}&\hdots&\Psi_{n-2;a}^{[1]}&\Psi_{n-1;a}\\
0&n-2&\Psi_{1;a}^{[n-3]}&\hdots&\Psi_{n-3;a}^{[1]}&\Psi_{n-2;a}\\
\vdots&\ddots&\ddots&\ddots&\vdots&\vdots\\
\vdots&\vdots&\ddots&2&\Psi_{1;a}^{[1]}&\Psi_{2;a}\\
0&\hdots&\hdots&0&1&\Psi_{1;a}
\end{vmatrix}.
\end{align}\end{proposition}
\begin{proof}These identities are proved similarly to \cite{Getal} by solving systems of linear equations arising from Lemma \ref{psieq}.
\end{proof}

We note that all $\Lambda_{n;a}$, $S_{n;a}$ are hence polynomials in the $\Psi_{l;a}$ with coefficients in $\mQ[a]$. The corresponding identities for noncommutative symmetric functions are recovered if $a=0$.

\begin{example}
\begin{align*}
\Psi_{2;a}&=-2\Lambda_{2;a}+\Lambda_{1;a}^{[1]}\Lambda_{1;a}=2S_{2;a}^{[1]}-S_{1;a}^{[1]}S_{1;a},\\
\Psi_{3;a}&=3\Lambda_{3;a}-2\Lambda_{2;a}^{[1]}\Lambda_{1;a}-\Lambda_{1;a}^{[2]}\Lambda_{2;a}+\Lambda_{1;a}^{[2]}\Lambda_{1;a}^{[1]}\Lambda_{1;a}\\
&=3S_{3;a}^{[2]}-2S_{2;a}^{[2]}S_{1;a}-S_{1;a}^{[2]}S_{2;a}^{[1]}+S_{1;a}^{[2]}S_{1;a}^{[1]}S_{1;a}.
\end{align*}
\end{example}


\subsection{A Hopf Algebra Structure}\label{hopf}

In this section, we propose a Hopf algebra structure on $\Sym^a$ by making the generators $\Psi_{n;a}$ primitive elements. Actions on other generators $S_{n;a}$, and $\Lambda_{n;a}$ can be computed using the translation formulas from Lemma \ref{psieq}. In the commutative case,  a Hopf algebra structure on $\Lambda^a$ was introduced in \cite{Mol2}*{Section 4}.

\begin{proposition}
We define a coproduct on $\Sym^a$ by requiring $\Psi_{n;a}$ to be primitive elements, i.e.
\begin{align}
\Delta(\Psi_{n;a})&=\Psi_{n;a}\otimes 1+1\otimes \Psi_{n;a}.
\end{align}
This way, $\Sym^a$ becomes a Hopf algebra, with antipode defined by $S(\Psi_{n;a})=-\Psi_{n;a}$.
\end{proposition}

It is in general not straightforward to compute the coproducts of $S_{n;a}$, and $\Lambda_{n;a}$. We restrict ourselves to include the following examples in the noncommutative setup:
\begin{align*}
\Delta(S_{2;a})&=S_{2;a}\otimes 1+\frac{1}{2}S_{1;a}^{[-1]}\otimes S_{1;a}+\frac{1}{2}S_{1;a}\otimes S_{1;a}^{[-1]}+1\otimes S_{2;a}\\
&=S_{2;a}\otimes 1+S_{1;a}\otimes S_{1;a}+1\otimes S_{2;a},\\
\Delta(S_{3;a})&=S_{3;a}\otimes 1+S_{2;a}\otimes S_{1;a}+S_{1;a}\otimes S_{2;a}+1\otimes S_{3;a}+\frac{4}{3}(a_0-a_1)S_{1;a}\otimes S_{1;a}.
\end{align*}

\subsection{Multiparameter Quasi-Schur Functions}\label{quasiSchur}

In this section, we define multiparameter quasi-Schur functions and present a version of Giambelli's formula. We note, as in \cite{Getal}*{Section 3.3}, that the quasi-Schur functions are \emph{not} elements of $\Sym^a$ but are contained in the skew-field freely generated by $S_{1;a},S_{2;a},S_{3;a},\ldots$. The treatment here is similar to the unshifted case in \cite{Getal}. 

\begin{definition}
Let $\lambda=(\lambda_1\leq \lambda_2\leq \ldots \leq \lambda_n)$ be a partition. We define the \emph{multiparameter quasi-Schur function} $\breve{S}_{\lambda;a}$ as the quasideterminant
\begin{align}
\breve{S}_{\lambda;a}=(-1)^{n-1}\begin{vmatrix}
S_{\lambda_1;a}^{[n-1]}&S_{\lambda_2+1;a}^{[n-1]}&\ldots&\boxed{S_{\lambda_n+n-1;a}^{[n-1]}}\\
S_{\lambda_1-1;a}^{[n-2]}&S_{\lambda_2;a}^{[n-2]}&\ldots&S_{\lambda_n+{n}-2;a}^{[n-2]}\\
\vdots&\vdots&\ddots&\vdots\\
S_{\lambda_1-{n}+1;a}&S_{\lambda_2-n+2;a}&\ldots&S_{\lambda_n;a}
\end{vmatrix}.
\end{align}
Here, we use the convention that $S_{-k;a}=0$ for $k>0$.
\end{definition}

This definition is similar to the formula from \cite{OO}*{Theorem 13.1}, for $a_i=i-1$,  and \cite{Mol2} for more general parameters $a$, in the commutative case. In the case $a_i=i-1$, if all $S_{k;a}=S_k^*$ commute, we obtain a ratio of shifted Schur functions $s_\lambda^*/s_{\lambda'}^*$, where $\lambda'=(\lambda_1-1,\ldots, \lambda_{n-1}-1)$.

\begin{example}
\begin{gather*}
\breve{S}_{k;a}=S_{k;a},\qquad  \breve{S}_{(1^k);a}=\Lambda_{k;a},\qquad
\breve{S}_{(1^2,3);a}=\begin{vmatrix}
S_{1;a}^{[2]}&S_{2;a}^{[2]}&S_{5;a}^{[2]}\\
1&S_{1;a}^{[1]}&S_{4;a}^{[1]}\\
0&1&S_{3;a}
\end{vmatrix}=R_{(1^2,3);a},\\
\breve{S}_{(2,1,2);a}=\begin{vmatrix}
S_{2;a}^{[2]}&S_{2;a}^{[2]}&S_{4;a}^{[2]}\\
S_{1;a}^{[1]}&S_{1;a}^{[1]}&S_{3;a}^{[1]}\\
1&1&S_{2;a}
\end{vmatrix}\neq R_{(2,1,2);a}.
\end{gather*}
\end{example}

A formula similar to Proposition \ref{ribboninL} appears naturally for $a$-shifted quasi-Schur functions, without the corrective shift. This is a noncommutative analogue for \cite{OO}*{Eq. (13.10)}, and an $a$-shifted version of \cite{Getal}*{Proposition 3.18}. For this, denote by $\lambda^{\sim}$ the \emph{conjugate} partition to $\lambda$, i.e. the partition whose diagram is obtained by interchanging the rows and columns of $\lambda$.

\begin{proposition}For any partition $\lambda=(\lambda_1\leq \lambda_2\leq \ldots \leq \lambda_n)$ we find
\begin{align}
\breve{S}_{\lambda^{\sim};a}=(-1)^{n-1}\begin{vmatrix}
\Lambda_{\lambda_{n};a}&\Lambda_{\lambda_{n-1}+1;a}^{[-1]}&\ldots &\boxed{\Lambda_{\lambda_{1}+{n}-1;a}^{[1-n]}}\\
\Lambda_{\lambda_{n}-1;a}&\Lambda_{\lambda_{n-1};a}^{[-1]}&\ldots &\Lambda_{\lambda_{1}+{n}-2;a}^{[1-n]}\\
\vdots&\vdots&\ddots&\vdots\\
\Lambda_{\lambda_{n}-n+1;a}&\Lambda_{\lambda_{n-1}-n+2;a}^{[-1]}&\ldots &\Lambda_{\lambda_{1};a}^{[1-n]}
\end{vmatrix}.
\end{align}
\end{proposition}

\begin{example}The case $\lambda=(1,1,2)$ gives $\lambda^{\sim}=(1,3)$ and
\begin{align*}
\breve{S}_{(1,1,2);a}&=(-1)\begin{vmatrix}
S_{1;a}^{[2]}&S_{2;a}^{[2]}&\boxed{S_{4;a}^{[2]}}\\
1&S_{1;a}^{[1]}&S_{3;a}^{[1]}\\
0&1&S_{2;a}
\end{vmatrix}=
\begin{vmatrix}
\Lambda_{3;a}&\boxed{\Lambda_{4;a}^{[-1]}}\\
1&\Lambda_{1;a}^{[-1]}
\end{vmatrix}.
\end{align*}
\end{example}

We can extend the anti-algebra morphism $\omega_a\colon \Sym^a\to \Sym^{\hat{a}}$ to an anti-algebra isomorphism of the associated skew-fields. It follows from the above proposition that
\begin{align}
\omega_a\left( \breve{S}_{\lambda;a} \right)=\breve{S}_{\lambda^{\sim};\hat{a}}.
\end{align}

Recall the \emph{Frobenius form} of a partition, i.e. a decomposition of $\lambda$ into hook sub-diagrams, cf. the notation of \cite{Getal}*{Section 3.3}. In particular, we denote the partition $(n|m):=(1^n,m+1)$.
We obtain the following analogue of Giambelli's formula for $a$-shifted quasi-Schur functions,  analogously to \cite{Getal}*{Proposition 3.20}. 

\begin{proposition}[Giambelli Formula]\label{giambelli}Given $\lambda$ in Frobenius notation as $(\beta_1,\ldots, \beta_k|\alpha_1,\ldots, \alpha_k)$,
\begin{align}
\breve{S}_{\lambda;a}&=\begin{vmatrix}
\breve{S}_{(\beta_1|\alpha_1);a}&\breve{S}_{(\beta_1|\alpha_2);a}&\ldots& \breve{S}_{(\beta_1|\alpha_k);a}\\
\breve{S}_{(\beta_2|\alpha_1);a}&\breve{S}_{(\beta_2|\alpha_2);a}&\ldots& \breve{S}_{(\beta_2|\alpha_k);a}\\
\vdots&\vdots&\ldots&\vdots\\
\breve{S}_{(\beta_k|\alpha_1);a}&\breve{S}_{(\beta_k|\alpha_2);a}&\ldots& \boxed{\breve{S}_{(\beta_k|\alpha_k);a}}
\end{vmatrix}.
\end{align}
\end{proposition}
\begin{proof}
This follows as in \cite{Getal}*{Proposition 3.20}, using a version of Bazin's theorem for quasideterminants. We write
\begin{align*}
\begin{vmatrix}
i_1&i_2&\ldots&\boxed{i_n}
\end{vmatrix}=
\begin{vmatrix}
S_{i_1;a}^{[n-1]}&S_{i_2;a}^{[n-1]}&\ldots&\boxed{S_{i_n;a}^{[n-1]}}\\
S_{i_1-1;a}^{[n-2]}&S_{i_2-1;a}^{[n-2]}&\ldots&S_{i_n-1;a}^{[n-2]}\\
\vdots&\vdots&\ldots&\vdots\\
S_{i_1-n+1;a}&S_{i_2-n+1;a}&\ldots&S_{i_n-n+1;a}
\end{vmatrix}.
\end{align*}
The statement is now proved as in \cite{Getal}*{Proposition 3.20}. We note that indeed
$$\begin{vmatrix}0&1&\ldots& k-2&k&k+1\ldots&n&\boxed{n+m+2}\end{vmatrix}
=\breve{S}_{(n-k|m)},$$
which follows using \cite{Getal}*{Proposition 2.12}.
\end{proof}

This formula generalizes \cite{OO}*{Eq. (13.17)}, which was remarked by G. Olshanski is independent of shifts. The same observation applies here.

\subsection{A Realization of \texorpdfstring{$\Sym^a$}{Sym-a} inside of \texorpdfstring{$\Sym$}{Sym}}

We can adapt the point of view from \cite{ORV2}*{Section \S 3} to the noncommutative setup and realize multiparameter ribbon Schur functions in $\Sym$ by requiring that
\begin{align*}
\sigma^{a}(t)&:=1+\sum_{k=1}^{\infty}{\frac{S_{k;a}}{\spower{t}{a}^k}}=1+\sum_{k=1}^{\infty}{\frac{S_k}{t^k}}=\sigma(1/t),\\
\lambda^{\hat{a}}(t)&:=1+\sum_{k=1}^{\infty}{\frac{\Lambda_{k; a}}{{\spower{t}{\hat{a}}^k}}}=1+\sum_{k=1}^{\infty}{\frac{\Lambda_k}{t^k}}=\lambda(1/t).
\end{align*}
These assumptions are not made in other sections of the paper, where we cannot relate the generating series of $a$-shifted symmetric functions to the unshifted ones.

In this setup, rather than defining a new ring $\Sym^a$, we can relate the unshifted symmetric functions to the $a$-shifted ones by more precise formulas.

\begin{lemma}\label{rewritelemma}For any series of parameters $a=(a_i)$,
\begin{align}
S_n&=\sum_{i=0}^{n} h_i(a_1,\ldots,a_{n-i})S_{n-i;a}, &
\Lambda_n&=\sum_{i=0}^{n} (-1)^i h_i(a_1,\ldots,a_{-n+i+2})\Lambda_{n-i;a}.\label{nonshiftintermsofshift}\\
S_{n;a}&=\sum_{i=0}^{n} (-1)^ie_i(a_1,\ldots,a_{n-1})S_{n-i}, &
\Lambda_{n;a}&=\sum_{i=0}^{n} e_i(a_1,\ldots,a_{-n+3})\Lambda_{n-i}.\label{shiftintermsofnonshift}
\end{align}
\end{lemma}
Here, $h_k(a_1,\ldots,a_l)$ denotes the $k$-th commutative complete homogeneous symmetric function evaluated at the sequence $(a_1,\ldots,a_l)$, and similarly $e_{k}(a_1,\ldots, a_n)$ for the elementary symmetric function. We use the convention that for $j>i$ in $h_j(a_1,\ldots,a_i)$ all terms involving more than $i$ variables are simply omitted and only terms that do not require more than $b$ distinct entries appear in the sum.
E.g $S_3(a_1,a_2)=a_1^3+a_2^3+a_1a_2^2+a_1^2a_2$.

\begin{proof}[Proof of Lemma \ref{rewritelemma}]
Use the geometric series to write
\begin{align*}
\frac{1}{t-a_k}=\frac{1}{t}\sum_{i\geq 0}\frac{a_k^i}{t^i},
\end{align*}
for any $k$. The results follows by comparing coefficients in Eq. (\ref{Eka}), respectively Eq. (\ref{Lka}). The inverted formulas follow similarly. 
\end{proof}
Note that in particular, that in the context of this subsection, 
\begin{align*}
S_{1;a}&=S_1,&
S_{2;a}&=S_2-a_1S_{1},&
S_{3;a}&=S_3-(a_1+a_2)S_{2}+a_1a_2S_{1}.
\end{align*}
Eqs. (\ref{nonshiftintermsofshift})--(\ref{shiftintermsofnonshift}) hold equally well in the commutative case, as the generating series have the same form.

This way, the multiparameter ribbon Schur functions form a nonhomogeneous $\mZ[a]$-bases for the algebra $\mZ[a]\otimes_{\mZ}\Sym_{\mZ}$. For this, we use define Schur functions as in Definition \ref{Schurdef}, and replace appearances of $S_{i;a}^{[s]}$ by $S_{k;\tau^{-s}a}$.

This setup particularly applies to Example \ref{expl3}, with parameters $a_i=i-1/2$. In the commutative setup, this was used in \cites{Mol,ORV2} to define Frobenius--Schur functions.

\section{Specialization}\label{sect3}

In this section, we define specializations of the noncommutative $a$-shifted symmetric functions $S_{k;a}$ and $\Lambda_{k;a}$ in terms of noncommutative rational functions in a list of indeterminants $(x_1,\ldots, x_n)$. Following \cite{Getal}*{Section~7}, this approach uses quasideterminants in order to obtain analogues of the expressions of the shifted commutative symmetric functions in terms of quotients of determinants in \cite{OO}*{Eq. (0.3)}. For this, we fix noncommuting indeterminants $x_1,\ldots, x_n$. We work in the free skew field generated by the $x_i$, cf. \cite{Getal}.

\subsection{Shifted Symmetric Specialization}\label{sect3.1}

We have to restrict generality for the constructions of this section, yielding deformations of noncommutative functions, but not for arbitrary sequences of shift-parameters.

\begin{assumption}\label{equiass}
In this section, assume that $a$ is \emph{equidistant}, i.e. there exists a constant $c\in \Bbbk$ such that $a_{n+k}-a_k=c\cdot n$ for all $n,k$ in $\mZ$. Examples \ref{expl1}--\ref{expl3} all satisfy this assumption.  
\end{assumption}

\begin{lemma}\label{equishift}
Let $a$ be equidistant, with $a_{n+k}-a_k=c\cdot n$ for all $n,k$ in $\mZ$. Then
\begin{align}
\abinom{k}{\nu}_k^{a}=c^\nu\binom{k}{\nu}\opower{(k+\nu-1)}{\nu}.
\end{align}
In fact, $\Sym^a$ is isomorphic to $\Sym^*$ as a graded algebra if $c\neq 0$, and $\Sym^a\cong \Sym$ if $c=0$.
\end{lemma}

This deformation point of view comes from \cite{OO}*{Remark 1.7} in the commutative case.

\begin{definition}
The elementary $a$-shifted symmetric functions $\Lambda_{k;a}(x_1,\ldots,x_n)$ are defined by setting
\begin{equation}\label{specializedseries}
\sum_{k\geq 0}\frac{\Lambda_{k;a}(x_1,\ldots,x_n)}{\spower{-t}{\hat{a}}^k}=\begin{vmatrix}
1&1&\ldots&1&1/\spower{t}{\tau^{-n} a}^n\\
\spower{x_1}{\tau^{1-n}a}^1&\spower{x_2}{\tau^{2-n}a}^1&\ldots&\spower{x_n}{a}^1&1/\spower{t}{\tau^{1-n} a}^{n-1}\\
\spower{x_1}{\tau^{1-n}a}^{2}&\spower{x_2}{\tau^{2-n}a}^2&\ldots&\spower{x_n}{a}^2&1/\spower{t}{\tau^{2-n} a}^{n-2}\\
\vdots&\vdots&\ldots&\vdots&\vdots\\
\spower{x_1}{\tau^{1-n}a}^n&\spower{x_2}{\tau^{2-n}a}^n&\ldots&\spower{x_n}{a}^n&\boxed{1}
\end{vmatrix}.
\end{equation}
That is, the right hand side is used to specify the generating series $\lambda^{\hat{a}}(-t)$.
\end{definition}

Observe that with this specialization, we can express $\Lambda_{k;a}(x_1,\ldots, x_n)$ as a quotient of quasideterminants.

\begin{theorem}[Specialization]\label{specialize} Let $a$ be equidistant. 
For $1\leq k\leq n$ we find that $\Lambda_{k;a}(x_1,\ldots,x_n)$ equals
\begin{equation*}
(-1)^{k-1}
\begin{vmatrix}
1&\ldots&1\\
\vdots&\ldots&\vdots\\
\spower{x_1}{\tau^{1-n}a}^{n-k-1}&\ldots&\spower{x_n}{a}^{n-k-1}\\
\spower{x_1}{\tau^{1-n}a}^{n-k+1}&\ldots&\spower{x_n}{a}^{n-k+1}\\
\vdots&\ldots&\vdots\\
\spower{x_1}{\tau^{1-n}a}^n&\ldots&\boxed{\spower{x_n}{a}^n}
\end{vmatrix}\cdot
{\begin{vmatrix}
1&\ldots&1\\
\vdots&\ldots&\vdots\\
\spower{x_1}{\tau^{1-n}a}^{n-k}&\ldots&\boxed{\spower{x_n}{ a}^{n-k}}\\
\vdots&\ldots&\vdots\\
\spower{x_1}{\tau^{1-n}a}^{n-1}&\ldots&\spower{x_n}{a}^{n-1}
\end{vmatrix}}^{-1}.
\end{equation*}
(The first quasi-minor is read in the way that the row with $(n-k)$-th $a$-shifted powers is deleted, and we start with the $0$-th $a$-shifted powers, unless $n-k=0$.)
Further, $\Lambda_{1;a}(x_1,\ldots, x_n)=1$ and $\Lambda_{k;a}(x_1,...,x_n)=0$ as soon as $k>n$.

Moreover, the $a$-shifted complete homogeneous symmetric function $S_{k;a}(x_1,\ldots,x_n)$ equals
\begin{equation*}
\begin{vmatrix}
1&\ldots&1\\
\vdots&\ldots&\vdots\\
\spower{x_1}{\tau^{1-n}a}^{n-3}&\ldots&\spower{x_n}{a}^{n-3}\\
\spower{x_1}{\tau^{1-n}a}^{n-2}&\ldots&\spower{x_n}{a}^{n-2}\\
\spower{x_1}{\tau^{1-n}a}^{n+k-1}&\ldots&\boxed{\spower{x_n}{a}^{n+k-1}}
\end{vmatrix}\cdot
{\begin{vmatrix}
1&\ldots&1\\
\vdots&\ldots&\vdots\\
\spower{x_1}{\tau^{1-n}a}^{n-1}&\ldots&\boxed{\spower{x_n}{a}^{n-1}}
\end{vmatrix}}^{-1},
\end{equation*}
and $S_{0;a}(x_1,\ldots, x_n)=1$, as well as $S_{k;a}(x_1,...,x_n)=0$ as soon as $k>n$
\end{theorem}

First, in order to derive the formula for $\Lambda_{k;a}(x_1,\ldots,x_n)$ we use the second expansion formula from \cite{Getal}*{Proposition 2.12} with $l=n$ and apply it to Eq. (\ref{specializedseries}). Then the parameter $1/\spower{t}{\tau^{-i} a}^{i}$ has coefficient given by  $|A^{{n-i+1},n+1}|_{n+1,n}|A^{n+1,n+1}|^{-1}_{n-i+1,n}$, where $i=1,\ldots, n$. The first minor $|A^{n-i+1,n+1}|_{n+1,n}$ deletes the last column and the row containing the $(n-i)$-th $a$-shifted powers and is the quasideterminant evaluated at the left bottom corner. Similarly, we recognize the second minor quasideterminant as stated.

The proof of the formula for $S_{k;a}$ uses a similar strategy to \cite{Getal}*{Proposition 7.5}. First, we introduce some simplifying notation. We want to verify that the $S_{k;a}(x_1,\ldots, x_n)$ defined by the formulas in the statement of this theorem satisfy Eq. (\ref{LintermsofS}). For this, we define
$$\dpower{s}{m_1,\ldots, \boxed{m_t},\ldots, m_n}=\begin{vmatrix}
\spower{x_1}{\tau^{1-s}a}^{m_1}&\spower{x_2}{\tau^{2-s}a}^{m_1}&\hdots& \spower{x_n}{\tau^{n-s}a}^{m_1}\\
\spower{x_1}{\tau^{1-s}a}^{m_2}&\spower{x_2}{\tau^{2-s}a}^{m_2}&\hdots& \spower{x_n}{\tau^{n-s}a}^{m_2}\\
\vdots&\vdots&\hdots&\vdots\\
\spower{x_1}{\tau^{1-s}a}^{m_t}&\spower{x_2}{\tau^{2-s}a}^{m_t}&\hdots&\boxed{\spower{x_n}{\tau^{n-s}a}^{m_t}}
\\
\vdots&\vdots&\hdots&\vdots\\
\spower{x_1}{\tau^{1-s}a}^{m_n}&\spower{x_2}{\tau^{2-s}a}^{m_n}&\hdots& \spower{x_n}{\tau^{n-s}a}^{m_n}\\
\end{vmatrix}.$$
Then we can write
\begin{align}
\Lambda_{k;a}(x_1,\ldots,x_n)&=\dpower{n}{0,\ldots n-k-1,n-k+1, \ldots, \boxed{n}}\cdot\dpower{n}{0,\ldots \boxed{n-k},\ldots n-1}^{-1},\label{appreviatedL}\\
S_{k;a}(x_1,\ldots,x_n)&=\dpower{n}{0,\ldots n-2, \boxed{n+k-1}}\cdot\dpower{n}{0,\ldots n-2, \boxed{n-1}}^{-1}.\label{appreviatedS}
\end{align}
Note that by elementary properties of quasideterminants, found in \cite{GGRW}*{Section 1.3}, for $i=1,\ldots, n-1$, multiplying the $i$-th column by $x_i-a_{i-s}$ on the right or left does not change the value of the quasideterminant. However, multiplying the $n$-th and last column by $x_n-a_{n-s}$ gives
\begin{align*}\dpower{s+1}{m_1+1,\ldots, \boxed{m_t+1},\ldots, m_n+1}
&=\dpower{s}{m_1,\ldots, \boxed{m_t},\ldots, m_n}(x_n-a_{n-s})\\
&=(x_n-a_{n-s})\dpower{s}{m_1,\ldots, \boxed{m_t},\ldots, m_n},
\end{align*}
where in the last equality we use that in the $n$-th column only shifts of $x_n$ appear and therefore left multiplication by 
$x_n-a_{n-s}$ equals right multiplication by this element.
Hence, we find the alternative expressions 
\begin{align}\label{altrel1}
\begin{split}\Lambda_{k;a}(x_1,\ldots,x_n)=&\dpower{n+s}{s,\ldots, n+s-k-1,n+s-k+1, \ldots, \boxed{n+s}}\\ &\cdot\dpower{n+s}{s,\ldots, \boxed{n+s-k},\ldots, n+s-1}^{-1},\end{split}
\\ \label{altrel2}
\begin{split}S_{k;a}(x_1,\ldots,x_n)=&\dpower{n+s}{s,\ldots, n+s-2, \boxed{n+s+k-1}}\\ &\cdot\dpower{n+s}{s,\ldots, n+s-2, \boxed{n+s-1}}^{-1},\end{split}
\end{align}
for any integer $s\geq 0$. 

\begin{definition}\label{variableshift}
The \emph{variable shifts} $\psi^{[s]} f(x_1,\ldots, x_n)$ are given by 
\begin{gather}
\psi^{[s]}(S_{k;a}(x_1,\ldots,x_n))=\dpower{n+s}{0,\ldots n-2, \boxed{n+k-1}}\cdot\dpower{n+s}{0,\ldots n-2, \boxed{n-1}}^{-1},\\
\begin{split}\psi^{[s]}(\Lambda_{k;a}(x_1,\ldots,x_n))=&\dpower{n+s}{0,\ldots n+k-1,n-k+1, \ldots, \boxed{n}}\\&\cdot\dpower{n+s}{0,\ldots \boxed{n-k},\ldots, n-1}^{-1}.\end{split}
\end{gather}
\end{definition}

As $a$ is equidistant, we see that 
\begin{equation}
\psi^{[s]}(S_{k;a}(x_1,\ldots, x_n))=S_{k;a}(x_1+sc,\ldots, x_n+sc).
\end{equation}
In particular, $\psi^{[s]}=\psi^s$, for $\psi^{[1]}=\psi$ in this case. 

Note that the shift $\psi$ does not coincide with the shift $\phi$ obtained using generating series in Section \ref{shiftsection}. However, we can describe the relationship between the two operations.

\begin{lemma}
If $a$ is equidistant, then the equation
\begin{align}\label{varshiftrel}\begin{split}
\psi S_{k;a}(x_1,\ldots, x_n)&=S_{k;a}(x_1+c,\ldots, x_n+c)\\
&=S_{k;a}(x_1,\ldots, x_n)+c(n+k-1)S_{k-1;a}(x_1,\ldots,x_n)\end{split}
\end{align}
holds for all $k\leq n$.
\end{lemma}

\begin{proof}
First observe the elementary equality
\begin{equation}\label{elementaryrec}
\spower{x}{a}^{n}=\spower{x}{\tau a}^n+(a_{n+1}-a_1)\spower{x}{\tau a}^{n-1}.
\end{equation}
We use this formula to find that
\[
\dpower{n+1}{0,\ldots n-2, \boxed{n-1}}=\dpower{n}{0,\ldots n-2, \boxed{n-1}}.
\]
Indeed, this follows as the $i$-th row 
\[
\begin{pmatrix}\spower{x_1}{\tau^{-n}a}^{i-1}&\spower{x_2}{\tau^{1-n}a}^{i-1}& \hdots & \spower{x_n}{\tau^{-1}a}^{i-1}\end{pmatrix}
\]
of the matrix on the left hand side equals 
\[
\begin{pmatrix}\spower{x_1}{\tau^{1-n}a}^{i-1}+(a_{i-n}-a_{1-n})\spower{x_1}{\tau^{1-n}a}^{i-2}& \hdots & \spower{x_n}{a}^{i-1}+(a_{i-1}-a_0)\spower{x_n}{a}^{i-2}\end{pmatrix}.
\]
As $a$ is equidistant, this is just the sum of the $i$-th row plus $c(i-1)$ times the $(i-1)$-st row. Using \cite{GGRW}*{Proposition 2.9} adding a row (which is not the $n$-th row) to another does not change $|A|_{nn}$. This proves the equality of the denominators of 
$S_{k;a}(x_1,\ldots, x_n)$ and $\psi S_{k;a}(x_1,\ldots, x_n)$.
The statement for higher shifts now follows from Lemma \ref{equishift}. 

Hence Equation (\ref{varshiftrel}) reduces to comparing the nominator quasideterminants. In fact, we shall  show that
\begin{align*}
\dpower{n+1}{0,\ldots n-2, \boxed{n+k-1}}&=\dpower{n}{0,\ldots n-2, \boxed{n+k-1}}\\&\phantom{=}+c(n+k-1)\dpower{n}{0,\ldots n-2, \boxed{n+k-2}}.
\end{align*}
Consider the left hand side. Using the same observation as above we can successively subtract $a_i-a_1$ times the $(i-1)$-th row from the $i$-th row as long as $i\leq n-1$, starting from the top. The $i$-th entry of the last row of $\dpower{n+1}{0,\ldots n-2, \boxed{n+k-1}}$ can be rewritten, using Eq. (\ref{elementaryrec}), as  
\begin{align*}
\spower{x_i}{\tau^{i-n}a}^{n+k-1}+c(n+k-1)\spower{x_i}{\tau^{i-n}a}^{n+k-2},&& i=1,\ldots, n.
\end{align*}
Now we use a row expansion of the rewritten quasideterminant, see the first formula of \cite{GGRW}*{Proposition 2.12}. This proves the stated formula.
\end{proof}

As a corollary of the above Lemma, we derive the formula
\begin{equation}\label{shiftsrelation}
\psi S_{k;a}(x_1,\ldots, x_n)=S_{k;a}(x_1+s,\ldots, x_n+s)=\phi^{[1]} S_{k;a}(x_1,\ldots, x_n)+ncS_{k-1;a}(x_1,\ldots, x_n),
\end{equation}
clarifying the relationship between the two different shifts $\phi$ and $\psi$.

\begin{proof}[Proof of Theorem \ref{specialize}]
We prove the theorem in two steps. First, we show that for any $n\geq 1$, the identity of quasideterminants
\begin{align}
\begin{split}
&\begin{vmatrix}
S_{1;a}^{[k-1]}&S_{2;a}^{[k-1]}&\hdots&S_{k-2;a}^{[k-1]}&S_{k-1;a}^{[k-1]}&\boxed{S_{k;a}^{[k-1]}}\\
1&S_{1;a}^{[k-2]}&\hdots&S_{k-3;a}^{[k-2]}&S_{k-2;a}^{[k-2]}&S_{k-1;a}^{[k-2]}\\
0&\ddots&\ddots&\vdots&\vdots&\vdots\\
\vdots&\hdots&\ddots&S_{1;a}^{[2]}&S_{2;a}^{[2]}&S_{3;a}^{[2]}\\
\vdots&\hdots&\hdots&1&S_{1;a}^{[1]}&S_{2;a}^{[1]}\\
0&\hdots&\hdots&0&1&S_{1;a}
\end{vmatrix}\\&=
\begin{vmatrix}
\psi^{k-1}(S_{1;a})&\psi^{k-1}(S_{2;a})&\hdots&\psi^{k-1}(S_{k-2;a})&\psi^{k-1}(S_{k-1;a})&\boxed{\psi^{k-1}(S_{k;a})}\\
1&\psi^{k-2}(S_{1;a})&\hdots&\psi^{k-2}(S_{k-3;a})&\psi^{k-2}(S_{k-2;a})&\psi^{k-2}(S_{k-1;a})\\
0&\ddots&\ddots&\vdots&\vdots&\vdots\\
\vdots&\hdots&\ddots&\psi^{2}(S_{1;a})&\psi^{2}(S_{2;a})&\psi^2(S_{3;a})\\
\vdots&\hdots&\hdots&1&\psi(S_{1;a})&\psi(S_{2;a})\\
0&\hdots&\hdots&0&1&S_{1;a}
\end{vmatrix}
\end{split}\label{qdetphipsi}
\end{align}
holds. Here, we regard $\psi$ as a formal linear operator on $\Sym^a$, which sends $S_{l;a}$ to $S_{l;a}+c(n+l-1) S_{l-1;a}$ for any $l$. 

By repeated use of Eqs. (\ref{varshiftrel}) we find that 
\begin{equation}
\psi^l(S_{k;a})=\sum_{\nu=0}^lc^\nu\binom{l}{\nu}\opower{(n+k-1)}{\nu}S_{k-\nu;a}.
\end{equation}
Further, we require the recursion
\begin{equation}\label{inductiverec}
\opower{(x+n)}{k}=\sum_{\nu=0}^n\binom{n}{\nu}\opower{k}{\nu}\opower{x}{k-\nu},
\end{equation}
which follows by repeated application of Eq. (\ref{elementaryrec}) with $c=1$.
We use these equations to derive 
\begin{align*}
c^t\binom{n}{t}\opower{l}{t}\psi^{l-t}(S_{k-t;a})&=c^t\binom{n}{t}\opower{l}{t}\sum_{\nu=0}^{l-t}c^{\nu}\binom{l-t}{\nu}\opower{(n-t+k-1)}{\nu}S_{k-t-\nu;a}\\
&=c^t\binom{n}{t}\opower{l}{t}\sum_{\nu=0}^{l-t}c^{\nu}\binom{l-t}{\nu}\sum_{\omega=0}^{n-t}\binom{n-t}{\omega}\opower{\nu}{\omega}\opower{(k-1)}{(\nu-\omega)}S_{k-t-\nu;a}\\
&=c^t\binom{n}{t}\opower{l}{t}\sum_{\nu=t}^{l}c^{\nu-t}\binom{l-t}{\nu-t}\sum_{\omega=t}^{n}\binom{n-t}{\omega-t}\opower{(\nu-t)}{(\omega-t)}\opower{(k-1)}{(\nu-\omega)}S_{k-\nu;a}\\
&=\binom{n}{t}\sum_{\nu=t}^{l}\binom{l}{\nu}c^\nu\sum_{\omega=t}^{n}\binom{n-t}{\omega-t}\opower{\nu}{t}\opower{(\nu-t)}{(\omega-t)}\opower{(k-1)}{(\nu-\omega)}S_{k-\nu;a}\\
&=\binom{n}{t}\sum_{\nu=t}^{l}\binom{l}{\nu}c^\nu\sum_{\omega=t}^{n}\binom{n-t}{\omega-t}\opower{\nu}{\omega}\opower{(k-1)}{(\nu-\omega)}S_{k-\nu;a}\\
&=\sum_{\nu=t}^l\binom{l}{\nu}c^\nu\sum_{\omega=t}^n\binom{n}{\omega}\opower{\nu}{\omega}\opower{(k-1)}{(\nu-\omega)}\binom{\omega}{t}S_{k-\nu;a}\\
&=\sum_{\nu=0}^l\binom{l}{\nu}c^\nu\sum_{\omega=0}^n\binom{n}{\omega}\opower{\nu}{\omega}\opower{(k-1)}{(\nu-\omega)}\binom{\omega}{t}S_{k-\nu;a}.
\end{align*}
Hence we derive, using that 
$$1=-\sum_{t=1}^\omega (-1)^{t}\binom{\omega}{t},$$
the recursion
\begin{align*}
\psi^l(S_{k;a})&=\sum_{\nu=0}^{l}c^{\nu}\binom{l}{\nu}\opower{(n+k-1)}{\nu}S_{k-\nu;a}\\
&=\sum_{\nu=0}^{l}c^\nu\binom{l}{\nu}\sum_{\omega=0}^n\binom{n}{\omega}\opower{\nu}{\omega}\opower{(k-1)}{(\nu-\omega)}S_{k-\nu;a}\\
&=S_{k;a}^{[l]}+\sum_{\nu=1}^{l}\binom{l}{\nu}c^\nu\sum_{\omega=1}^n\binom{n}{\omega}\opower{\nu}{\omega}\opower{(k-1)}{(\nu-\omega)}S_{k-\nu;a}\\
&=S_{k;a}^{[l]}-\sum_{\nu=1}^{l}c^\nu\binom{l}{\nu}\sum_{\omega=1}^n\sum_{t=1}^{\omega}(-1)^t\binom{n}{\omega}\opower{\nu}{\omega}\opower{(k-1)}{(\nu-\omega)}\binom{\omega}{t}S_{k-\nu;a}\\
&=S_{k;a}^{[l]}-\sum_{t=1}^{l}(-1)^t\sum_{\nu=0}^{l}c^\nu\binom{l}{\nu}\sum_{\omega=0}^n\binom{n}{\omega}\opower{\nu}{\omega}\opower{(k-1)}{(\nu-\omega)}\binom{\omega}{t}S_{k-\nu;a}\\
&=S_{k;a}^{[l]}-\sum_{t=1}^{l}(-1)^tc^t\binom{n}{t}\opower{l}{t}\psi^{l-t}(S_{k-t;a}).
\end{align*}
Setting $c_{l,t}:=(-1)^tc^t\binom{n}{t}\opower{l}{t}$ we see that
\begin{equation}
\psi^l(S_{k;a})=S_{k;a}^{[l]}-\sum_{t=1}^l c_{l,t}\psi^{t}(S_{k-t;a})
\end{equation}
in which the coefficients $c_{l,t}$ are independent of $k$. This relationship implies Eq. (\ref{qdetphipsi}) by  successively subtracting linear combinations of the lower rows to change the top rows (cf. properties of quasideterminants in \cite{GGRW}*{Section 1.3}). 

Finally, we proceed exactly as in the proof of \cite{Getal}*{Proposition 7.5}, based on a noncom\-mu\-ta\-tive version of Bazin's Theorem.

Consider the $2n\times n$-matrix defined by having entries
\begin{align*}
a_{ij}=\spower{x_j}{\tau^{1+j-n-k}a}^{i-1}, &&i=1,\ldots, 2n, j=1,\ldots, n.
\end{align*}
Then we can form a $k\times k$-matrix $B$ of quasideterminants as in Theorem \ref{bazin}. It has entries
\begin{align*}
b_{ij}=\dpower{n+k-1}{i-1,\ldots, i+n-3, \boxed{n+j-1}\,}, && i,j=1,\ldots,k.
\end{align*}
This equals $|A_{\lbrace i,\ldots, i+n-2, n+j\rbrace}|_{n+j,n}$ in the notation of Theorem \ref{bazin}. Therefore, we obtain by the same theorem that 
\begin{align*}
|B|_{1k}=&\dpower{n+k-1}{k-1,\ldots, n-2,n,\ldots\boxed{n+k-1}\,} \\
&\cdot \dpower{n+k-1}{k-1,\ldots,\boxed{n-1},\ldots,n+k-2\,}^{-1}\cdot \dpower{n+k-1}{0,1,\ldots,\boxed{n-1}\,} \\
=&\Lambda_{k;a}(x_1,\ldots, x_n)\cdot \dpower{n+k-1}{0,1,\ldots,\boxed{n-1}\,} 
\end{align*}
where we use Equation (\ref{altrel1}) for the last equality.
We can multiply the $i$-th row of $B$ by the factor $$\dpower{n+k-1}{i-1,i,\ldots,\boxed{n+i-2}\,}^{-1}$$
on the right to obtain  a new matrix $C$. Computational rules of quasideterminants as in \cite{GGRW}*{Proposition 2.9}, together with the above computation, give that $$|C|_{1k}=|B|_{1k}\cdot \dpower{n+k-1}{0,1,\ldots,\boxed{n-1}\,}^{-1}=\Lambda_{k;a}(x_1,\ldots, x_n). $$
Note finally that the entries of $C$ are given by
\begin{align*}
c_{ij}=&\dpower{n+k-1}{i-1,\ldots, i+n-3,\boxed{n+j-1}\,}
\cdot\dpower{n+k-1}{i-1,i,\ldots,\boxed{n+i-2}\,}^{-1}\\
=&\dpower{n+k-i}{0,\ldots, n-2,\boxed{n+j-i}\,}
\cdot\dpower{n+k-i}{0,1,\ldots,\boxed{n-1}\,}^{-1}\\
=&\begin{cases}\psi^{k-i} S_{j-i+1;a}(x_1,\ldots,x_n)& \text{if }j+1\geq i,\\
0& \text{if }j+1<i,
\end{cases}
\end{align*}
for $i,j=1,\ldots,k$. Hence, omitting the list $(x_1,\ldots,x_n)$ for brevity in the $S_{i;a}$, and $\Lambda_{k;a}$, we find that
\begin{align}\label{altLintermsofS}
\Lambda_{k;a}=(-1)^{k-1}\begin{vmatrix}
\psi^{k-1}(S_{1;a})&\psi^{k-1}(S_{2;a})&\hdots&\psi^{k-1}(S_{k-2;a})&\psi^{k-1}(S_{k-1;a})&\boxed{\psi^{k-1}(S_{k;a})}\\
1&\psi^{k-2}(S_{1;a})&\hdots&\psi^{k-2}(S_{k-3;a})&\psi^{k-2}(S_{k-2;a})&\phi^{k-2}(S_{k-1;a})\\
0&\ddots&&\ddots&\vdots&\vdots\\
\vdots&\hdots&\ddots&\psi^2(S_{1;a})&\psi^{2}(S_{2;a})&\psi^{2}(S_{3;a})\\
\vdots&\hdots&\hdots&1&\psi(S_{1;a})&\psi(S_{2;a})\\
0&\hdots&\hdots&0&1&S_{1;a}
\end{vmatrix}.
\end{align}
This completes the proof, using Eq. (\ref{qdetphipsi}).
\end{proof}

\begin{example} Consider the case $a_i=i-1$. Expansion of the quasideterminant quotients give, for small values of $n$ and $k$, that
\begin{gather*}
\Lambda_1^*(x_1)=S_1^*(x_1)=x_1,\\
\Lambda_1^*(x_1,x_2)=S_1^*(x_1,x_2)=(x_2(x_2-1)-(x_1+1)x_1)(x_2-x_1-1)^{-1},\\
\Lambda_2^*(x_1,x_2)=(x_2(x_2-1)-x_1x_2)((x_1+1)^{-1}x_2-1)^{-1},\\
S_2^*(x_1,x_2)=(x_2(x_2-1)(x_2-2)-(x_1+1)x_1(x_1-1))(x_2-x_1-1)^{-1}.
\end{gather*}
We can observe a symmetry here: The functions $\Lambda_1^*(x_1,x_2)$, $\Lambda_2^*(x_1,x_2)$ and $S_2^*(x_1,x_2)$ are stable under the transformation exchanging
$$(x_1,x_2)\longleftrightarrow(x_2-1,x_1+1).$$
\end{example}

\begin{remark}
In the case when all variables $x_i$ commute (assuming invertibility of the relevant terms), one recovers the shifted symmetric functions from \cite{OO}*{Definition 1.2, Eq. (1.10), (1.11)}. Note that the same does not hold for the ribbon Schur functions, as the phenomenon of recovering the commutative Schur functions from ribbon Schur functions already fails in the unshifted case. However, the following Proposition holds. 
\end{remark}

\begin{proposition}\label{recover}
If the variables $x_1,\ldots, x_n$ commute, then 
\begin{align*}
S_{k}^*(x_1,\ldots, x_n)=h_k^*(x_1,\ldots, x_n), && \Lambda_{k}^*(x_1,\ldots, x_n)=e_k^*(x_1,\ldots, x_n),
\end{align*}
where $h_k^*$, $e_k^*$ denote the shifted symmetric functions of \cite{OO}.
\end{proposition}
\begin{proof}
If all entries of $A$ commute, then $|A|_{pq}=(-1)^{p+q}\det(A)/\det(A^{pq})$ according to \cite{Getal}*{Proposition 2.3}. Consider the formulas for $\Lambda_k^*(x_1,\ldots, x_n)$ and $S_k^*(x_1,\ldots, x_n)$ from Theorem \ref{specialize}. For $\Lambda_k^*(x_1,\ldots, x_n)$, write $\dpower{n}{0,\ldots n-k-1,n-k+1, \ldots, n}$ for the matrix used in the denominator and $\dpower{n}{0,\ldots n-k,\ldots n-1}$ for the matrix used in the denominator in Theorem \ref{specialize} for $\Lambda_k^*$. Then
\begin{align*}
\det(\dpower{n}{0,\ldots n-k-1,n-k+1, \ldots, n}^{n,n})=\det(\dpower{n}{0,\ldots n-k,\ldots n-1}^{n-k,n}).
\end{align*}
Hence 
\begin{align*}
\Lambda_k^*(x_1,\ldots, x_n)=
\frac{\det\begin{pmatrix}
1&\ldots&1\\
\vdots&\ldots&\vdots\\
\opower{(x_1+n-1)}{(n-k-1)}&\ldots&\opower{x_n}{(n-k-1)}\\
\opower{(x_1+n-1)}{(n-k+1)}&\ldots&\opower{x_n}{(n-k+1)}\\
\vdots&\ldots&\vdots\\
\opower{(x_1+n-1)}{n}&\ldots&\opower{x_n}{n}
\end{pmatrix}}
{\det \begin{pmatrix}
1&\ldots&1\\
\vdots&\ldots&\vdots\\
\opower{(x_1+n-1)}{(n-k)}&\ldots&\opower{x_n}{(n-k)}\\
\vdots&\ldots&\vdots\\
\opower{(x_1+n-1)}{(n-1)}&\ldots&\opower{x_n}{(n-1)}
\end{pmatrix}}.
\end{align*}
This equals $h_k^*(x_1,\ldots, x_n)$ as defined in \cite{OO}*{Eq. (0.3)} after inverting the order of the rows in nominator and denominator. The proof for $S_k^*$ is similar.
\end{proof}

\begin{proposition}[Shifted Symmetry]\label{shiftedsymmetries}
The functions $\Lambda_{k;a}(x_1,\ldots, x_n)$ and  $S_{k;a}(x_1,\ldots, x_n)$, and hence all multiparameter ribbon Schur functions, are symmetric under exchange of $$(x_i,x_{i+1})\longleftrightarrow (x_{i+1}-c,x_i+c).$$
\end{proposition}
\begin{proof}
Note that, by \cite{GGRW}*{Proposition 1.2.4}, such a change of variables does not change the quasideterminant in the right hand side of Eq. (\ref{specializedseries}). Hence, it does not change the functions $\Lambda_{k;a}(x_1,\ldots, x_n)$. The functions $S_{k;a}(x_1,\ldots, x_n)$ are hence invariant under the same change of variables using Theorem \ref{specialize}, noting that they are polynomials in (shifts of) $\Lambda_{k;a}(x_1,\ldots, x_n)$.
\end{proof}

\subsection{Stability under Extension}\label{sect3.2}

Let $X=\lbrace x_1,x_2,\ldots\rbrace$ be noncommuting variables, and write $X_n=\lbrace x_1,\ldots, x_n\rbrace$. Denote by $\Sym^a(X_n)$ the subalgebra of the free skew-field in the variables $X_n$ generated by $\lbrace \Lambda_{k;a}(x_1,\ldots,x_n)\rbrace_{k\geq 0}$. 
Generalizing \cite{Getal}*{Section 7}, we can realize $\Sym^a$ as the projective limit  of the algebras $\Sym^a(X_n)$.

\begin{theorem}[Extension Stability]
There is an isomorphism of algebras 
$$\Sym^a\cong \varprojlim\Sym^a(X_n).$$
\end{theorem}
\begin{proof}
This follows using the projections $\Sym^a(X_{n+\nu})\to \Sym^a(X_n)$, given by 
$$f(x_1,\ldots, x_{n+\nu})\longmapsto f(x_1,\ldots, x_{n},a_1,\ldots, a_1),$$
which are morphisms of algebras. 

In fact, it suffices to show this statement for $\nu=1$ and $f=S_{k;a}$. Consider $S_{k;a}(x_1,\ldots, x_n,a_1)$. The last column of both nominator $|A|_{n+1,n+1}$ and denominator $|B|_{n+1,n+1}$ are quasideterminants consisting of zeros, and a single $1$ in the top right corner. We expand both of the quasideterminants using the second formula in \cite{Getal}*{Proposition 12}, with $l=n$. This yields
\begin{align*}
S_{k;a}(x_1,\ldots, x_n,a_1)&=(-1)^{k-1}|A|_{n+1,n+1}|B|_{n+1,n+1}^{-1}\\
&=(-1)^{k-1}|A^{1,n+1}|_{n+1,n}|A^{n+1,n+1}|_{1,n}^{-1}(|B^{1,n+1}|_{n+1,n}|B^{n+1,n+1}|_{1,n}^{-1})^{-1}\\
&=(-1)^{k-1}|A^{1,n+1}|_{n+1,n}|B^{1,n+1}|_{n+1,n}^{-1}\\
&=(-1)^{k-1}\dpower{n+1}{1,\ldots, n-1,\boxed{n+k}}\cdot \dpower{n+1}{1,\ldots, n-1,\boxed{n}}\\
&=(-1)^{k-1}S_{k;a}(x_1,\ldots, x_n).
\end{align*}
Here, the third equality uses that $A$ and $B$ only differ in the last row, and the forth equality uses Eq. (\ref{altrel2}).
\end{proof}

We remark that $\Sym^a$ using this specialization as rational functions in infinitely many noncommuting variables has a filtration, given by the maximum degree of terms, where $\deg(x_i)=1$, $\deg(x_i^{-1})=-1$. The associated graded algebra is isomorphic to the algebra on noncommutative symmetric functions $\Sym$.

We further note that using the variable shift $y_i:=x_i-ic$, $\Sym^a(X_n)$ consists of noncommutative symmetric functions in the variables $Y_n=\lbrace y_1,\ldots,y_n\rbrace$. This defines an  isomorphism $\Sym^a(X_n)\cong \Sym(Y_n)$. However, such an isomorphism does not exist in the limit (cf. \cite{OO}*{Section 1}).

\appendix
\section{A Version of Bazin's Theorem}

In the proof of Theorem \ref{specialize} we apply the following version of Bazin's theorem for quasideterminants.

\begin{theorem}[Krob--Leclerc]\label{bazin}
Let $k\leq n$ be positive integers and $A$ a generic $2n\times n$-matrix. We denote by $A_{\lbrace j_1,\ldots, j_n\rbrace}$ the $n\times n$-submatrix consisting of the rows $j_1,\ldots, j_n$. Define a $k\times k$-matrix $B$ of quasideterminants by
\begin{align*}
b_{ij}&=|A_{\lbrace i,\ldots, i+n-2, n+j\rbrace}|_{n+j,n}, &\forall 1\leq i,j\leq k.
\end{align*}
Then
\begin{align}
|B|_{1k}&=|A_{\lbrace k,\ldots, n-1,n+1,\ldots, n+k\rbrace}|_{n+k,n}\cdot |A_{\lbrace k, \ldots, n+k-1\rbrace}|_{n,n}^{-1}\cdot |A_{\lbrace 1,\ldots, n\rbrace}|_{nn}.
\end{align}
\end{theorem}
Note that this is a different (transposed) version of Bazin's theorem which was derived in \cite{KL}*{Theorem 2.8}, using a $n\times 2n$-matrix instead.



\bibliography{biblio}
\bibliographystyle{amsrefs}

\end{document}